\documentclass[a4paper,12pt]{article}

\usepackage[english]{babel}
\usepackage[hmargin=2.5cm,vmargin=2.5cm]{geometry}
\usepackage{indentfirst}
\usepackage{hyperref}
\usepackage{bm,tikz}
\usetikzlibrary{positioning}
\usepackage[font=small]{caption}
\usepackage{amsmath,amssymb,amsthm,mathtools}
\usepackage{braket}
\allowdisplaybreaks[2]

\usepackage{enumitem}
\usepackage{booktabs,array}

\usepackage{graphicx}
\usepackage{tikz}
\usetikzlibrary{
  calc,
  matrix,
  decorations.pathmorphing,
  decorations.text,
  decorations.markings
}
\tikzset{
  every node/.style={minimum size=0pt, inner sep=1pt},
  ->-/.style={
    decoration={markings, mark=at position #1 with {\arrow{latex}}},
    postaction={decorate}
  }
}

\usepackage{xcolor}
\usepackage{cancel}
\usepackage{centernot}
\usepackage[normalem]{ulem}
\usepackage{authblk}
\usepackage{caption,subcaption}

\theoremstyle{plain}
\newtheorem{theorem}{Theorem}[section]

\newtheorem{thmA}{Theorem}

\newtheorem{lemma}[theorem]{Lemma}
\newtheorem{lem}[theorem]{Lemma}
\newtheorem{lemA}[thmA]{Lemma}

\newtheorem{corollary}[theorem]{Corollary}
\newtheorem{cor}[theorem]{Corollary}

\newtheorem{conj}[theorem]{Conjecture}
\newtheorem{conjA}[thmA]{Conjecture}

\newtheorem{obs}[theorem]{Observation}
\newtheorem{definition}[theorem]{Definition}

\newtheorem{remark}[theorem]{Remark}
\theoremstyle{definition}

\newcommand{\Z}{\mathbb{Z}}

\newcommand\F{\mathcal{F}}
\renewcommand\O{\mathcal{O}}

\newcommand\G{\mathcal{G}}

\renewcommand\le\leqslant
\renewcommand\leq\leqslant
\renewcommand\ge\geqslant
\renewcommand\geq\geqslant
\renewcommand\a{\alpha}
\renewcommand\b{\beta}
\newcommand\ord{\textrm{ord}}

\def\aftermath{\par\vspace{-\belowdisplayskip}\vspace{-\parskip}\vspace{-\baselineskip}}

\usepackage[square,sort,comma,numbers]{natbib}
\usepackage[nameinlink]{cleveref}

\newcommand{\customfootnote}[1]{%
  \bgroup
  \renewcommand\thefootnote{}%
  \footnote{#1}%
  \addtocounter{footnote}{-1}%
  \egroup
}

\usetikzlibrary{decorations.markings, arrows.meta}
\tikzset{
  mid arrow/.style={
    postaction={
      decorate,
      decoration={
        markings,
        mark=at position .7 with {\arrow{Stealth}}
      }
    }
  }
}
\tikzset{
  lmid arrow/.style={
    postaction={
      decorate,
      decoration={
        markings,
        mark=at position .625 with {\arrow{Stealth}}
      }
    }
  }
}

\begin{document}
\title{Reconfiguration of Nowhere-zero Flows}
\renewcommand{\thefootnote}{\fnsymbol{footnote}}

\newcommand{\nankaiaffil}{School of Mathematical Sciences and LPMC, Nankai University, Tianjin 300071, China; \texttt{\{%
\href{mailto:lijiaao@nankai.edu.cn}{lijiaao}, %
\href{mailto:wangzhou@nankai.edu.cn}{wangzhou}%
    \}@nankai.edu.cn} and 
    \texttt{\{%
\href{mailto:suboll@mail.nankai.edu.cn}{suboll},%
\href{mailto:xny@mail.nankai.edu.cn}{xny}%
  \}@mail.nankai.edu.cn};
    Research of JL supported by National Natural Science Foundation of China (No. 12571371) and Natural Science Foundation of Tianjin (No. 24JCJQJC00130); Research of ZW supported by National Natural Science Foundation of China (No. 12301444).}

\author{\normalsize
  Daniel W. Cranston\thanks{William \& Mary, Department of Mathematics; Williamsburg, VA, USA; 
 \texttt{
\href{mailto:dcransto@gmail.com}{dcransto@gmail.com}}}, 
  Jiaao Li\textsuperscript{\textdagger}, 
  Bo Su\textsuperscript{\textdagger},\\ 
  Zhouningxin Wang\textsuperscript{\textdagger}, 
  and 
  Ningyan Xu\thanks{\nankaiaffil}
}
\renewcommand{\thefootnote}{\arabic{footnote}}
\setcounter{footnote}{0}

\date{}
\maketitle
\begin{abstract}
Fix an abelian group $A$, a graph $G$, and nowhere-zero $A$-flows $f'$ and $f''$ on $G$. Now $f'$ and $f''$ are \emph{$A$-flow-adjacent} if there exists a cycle $C$ in $G$ such that $f'(e)-f''(e)=0$ for all edges $e\notin E(C)$. 
And $f'$ and $f''$ are \emph{$A$-flow-equivalent} if there exists a sequence $f_0,\ldots,f_s$ of $A$-flows such that $f_0=f'$, $f_s=f''$, and $f_i$ and $f_{i-1}$ are $A$-flow-adjacent for all $i\in[s]$.
Given a group $A$, we seek conditions on a graph $G$ such that all $A$-flows on $G$ are pairwise $A$-flow-equivalent; in this case, we say that $G$ is \emph{$A$-flow-connected}. Analogously, we define $k$-flow-connectedness for nowhere-zero (integer) $k$-flows. The notions of $A$-flow-connectedness and $k$-flow-connectedness were first investigated by Esperet et al., who showed, among other results, that every $2$-edge-connected graph is $A$-flow-connected whenever $A=\mathbb{Z}_2^8$ or $|A| \ge 1.15\times 10^{694}$. 

In this paper, we first characterize the graphs that are $\Z_3$-flow-connected and that are $3$-flow-connected. We show that every 2-edge-connected graph is $A$-flow-connected if and only if this is true for every 2-edge-connected \emph{cubic} graphs. We show that all cubic bipartite graphs are $\Z_4$-flow-connected, and construct other cubic graphs that are and are not $\Z_4$-flow-connected. We conjecture that every Eulerian graph is $k$-flow-connected and $A$-flow-connected whenever $k$ or $|A|$ is at least 4; and provide evidence for this conjecture. Finally, we consider $4$-edge-connected graphs $G$. Here, we show that $G$ is $A$-flow-connected whenever $|A|\ge 5.3\times 10^6$.
\end{abstract}

\section{Introduction}
\label{intro-sec}
Fix a graph $G$ and an abelian group\footnote{All abelian groups in this paper are finite, except when we explicitly specify the integers $\mathbb{Z}$. Sometimes we state this explicitly, but not always.} $A$.
An \emph{$A$-flow} on $G$ is an ordered pair $(D,f)$, consisting of an orientation $D$ of $G$ and a function $f:E(G)\to A$ such that
$\sum_{e\in E_D^+(v)} f(e)
-
\sum_{e\in E_D^-(v)} f(e)
=0_A$
for every vertex $v$, where $E_D^+(v)$ and $E_D^-(v)$ denote, respectively, the sets of edges directed out of and into $v$.
Informally, we require that ``flow in equals flow out'' at each vertex.  An $A$-flow is \emph{nowhere-zero} if $f(e)\ne 0_A$ for each edge~$e$.  Tutte defined nowhere-zero flows in the 1940s and proved
that every planar
graph is $k$-face-colorable if and only if it has a nowhere-zero $\Z_k$-flow.  Motivated in part by this
duality, he conjectured\footnote{In fact,  Tutte phrased his conjectures in terms of integer-valued (rather than group-valued) flows, which we define soon.  But he also proved  the equivalence of the two formulations.} 
the following.

\begin{conj} The 3 statements below all hold.
	\begin{itemize}
		\item Every 2-edge-connected graph has a nowhere-zero $\Z_5$-flow.
		\item Every 2-edge-connected graph with no Petersen minor has a nowhere-zero $\Z_4$-flow.
		\item Every 4-edge-connected graph has a nowhere-zero $\Z_3$-flow.
	\end{itemize}
\end{conj}
These $3$ conjectures are all more than 50 years old, and nowhere-zero flows have played a central role in
structural and chromatic graph theory since they were posed. Despite some beautiful partial results~\cite{jaeger,seymour,LTWZ}, all $3$ conjectures remain open.

Given a graph $G$ and two proper $k$-colorings $\a$ and $\b$, the problem of \emph{$k$-coloring reconfiguration}
seeks to transform $\a$ to $\b$ by changing the color of one vertex at a time, requiring that each intermediate $k$-coloring is again proper. 
This topic has close ties to Glauber dynamics and sampling a random proper $k$-coloring.  Work in this area goes back many years, but it has received a more unified
treatment starting with the dissertation of Cereceda~\cite{cereceda-diss}.

Recently, Esperet et al.~\cite{esperet-et-al} initiated the study of reconfiguration of nowhere-zero flows. Now in each reconfiguration step, we may change the flow values on some cycle; here and throughout a \emph{cycle} is a connected 2-regular graph. (This definition extends
Tutte's duality in planar graphs between flows and face colorings.) Given a graph $G$, an abelian group $A$, and $A$-flows $\a,\b$, we again ask whether we can reconfigure $\a$ into $\b$. More generally, we seek conditions on $G$ and $A$ that allow us to transform every $A$-flow on $G$ into every other $A$-flow.  This motivates the following
definitions.

We consider two $A$-flows $(D,f)$ and $(D',f')$ to be the \emph{same} if for every edge $e$ with $D(e)=D'(e)$ we have $f(e)=f'(e)$, and for every edge $e$ with $D(e)=-D'(e)$ we have $f(e)+f'(e)=0_A$.
Due to the definition of ``same'', the orientation of a flow is typically unimportant, and we generally omit it; but we do consider the orientation fixed (across different flows). 
Nearly all flows considered in this paper are nowhere-zero; therefore, we typically \textbf{write $\bm{A}$-flows or $\bm{k}$-flows (defined below), and will explicitly indicate when flows are allowed to take the value 0.}
Two $A$-flows $f'$ and $f''$ on a graph $G$ are \emph{$A$-flow-adjacent} if there exists a cycle $C$ in $G$ such that $f'(e)=f''(e)$ for all $e\notin E(C)$.
And $f'$ and $f''$ are \emph{$A$-flow-equivalent} if there exists a sequence $f_0,\ldots,f_s$, with $f'=f_0$ and $f''=f_s$, of $A$-flows on $G$ such that $f_i$ and $f_{i-1}$ are $A$-flow-adjacent for each $i\in [s]$.  A graph $G$ is \emph{$A$-flow-connected} if all $A$-flows on $G$ are pairwise $A$-flow-equivalent.  
Now the \emph{$A$-flow reconfiguration graph} of $G$, denoted $\mathcal{F}(G,A)$, has as its vertices the 
$A$-flows on $G$, and two vertices are adjacent if and only if the corresponding flows are $A$-flow-adjacent.\footnote{If $G$ has no $A$-flow, i.e. $\F(G,A)$ is empty, then we say that $G$ is indeed $A$-flow-connected.  But this is a minor point, and it will rarely matter.} 

For $A=\mathbb Z$, a \emph{nowhere-zero $k$-flow} is an integer flow $(D,f)$ satisfying $0<|f(e)|<k$ for all $e\in E(G)$.
For each positive integer $k$, we analogously define \emph{$k$-flow-adjacency}, \emph{$k$-flow-equivalence}, \emph{$k$-flow-connectedness}, and \emph{$k$-flow reconfiguration graph} $\mathcal{F}(G,k)$.  

Given a $k$-flow on a graph $G$, it is easy to get a $\Z_k$-flow by taking each flow value modulo $k$.  The other direction is less obvious, but Tutte proved it: a graph $G$ has a $\Z_k$-flow if and only if $G$ has a $k$-flow; see Lemma~\ref{Zk-k-equiv-lem}. It is therefore natural to ask whether the same equivalence holds for flow reconfiguration. Esperet et al. believed that it does.

\begin{conjA}[\cite{esperet-et-al}]
For each graph $G$ and each positive integer $k$, the graph $\F(G,k)$ is connected if and only if the graph $\F(G,\Z_k)$ is connected.
\label{conjA}
\end{conjA}

One direction of this conjecture is easy: Every walk in $\F(G,k)$ maps to a walk in $\F(G,\Z_k)$. But the other direction remains open. Tutte proved that if $A$ and $B$ are abelian groups with $|A|=|B|$, then $G$ has an $A$-flow if and only if $G$ has a $B$-flow; more generally, $G$ has the same number of $A$-flows and $B$-flows. But the situation for flow-reconfiguration is more complex. Esperet et al.~\cite{esperet-et-al} constructed graphs $G$ such that $\F(G,\Z_2\times \Z_2)$ is connected, but $\F(G,\Z_4)$ is disconnected. The complete graph $K_4$ provides the simplest example, but there are many more.

Motivated by Tutte's $5$-flow conjecture, we seek to determine for which groups $A$ it is 
true that every $2$-edge-connected graph is $A$-flow. Esperet et al. \cite{esperet-et-al} showed that this is true when $A=\Z_2^8$. They also proved that it is true for all abelian groups $A$ that are sufficiently large; the bound arising from their proof is $|A|\ge 1.15\times 10^{694}$.

\subsection*{Our Contributions}

Our focus in this paper is on
determining conditions under which a graph $G$ is (or is not) $A$-flow-connected for certain groups $A$.
We will also say a bit about being $k$-flow-connected for positive integers $k$, although that
topic is less well-understood.

In Section~\ref{Z3-sec}, we characterize graphs that are $\Z_3$-flow-connected and graphs that are $3$-flow-connected. In Section~\ref{Z4flow-sec}, we study $\Z_4$-flow-connectedness of cubic graphs. We show that every cubic bipartite graph $G$ is $4$-flow-connected, and thus $\Z_4$-flow-connected.  We also construct other cubic graphs that are and are not $\Z_4$-flow-connected. In Section~\ref{sec:Eulerian}, we prove that every Eulerian graph is $2k$-flow-connected for every positive integer $k$, and conjecture that every Eulerian graph is $A$-flow-connected whenever $|A|\ge 4$. 

In Section~\ref{AFRC-sec}, we introduce a type of ``reducible'' configuration $H$ for being $A$-flow-connected. A graph $H$ is \emph{reducible} if each graph $G$ is $A$-flow-connected if and only if $G/H$ is $A$-flow-connected. Using this tool, we prove that each $n$-vertex graph is $A$-flow-connected whenever $|A|=\Omega(\log n)$.
We also show, for each abelian group $A$, that all 2-edge-connected graphs are $A$-flow-connected if and only if this is true for all 2-edge-connected \emph{cubic} graphs.
In Section~\ref{4edge-sec}, we consider 4-edge-connected graphs.  By adapting the above-mentioned work of Esperet et al. \cite{esperet-et-al}, we prove that every 4-edge-connected (or Eulerian) graph is $A$-flow-connected whenever $|A|\ge 5,251,339$.

\section{Characterizing graphs that are \texorpdfstring{$\Z_3$}{Z3}-Flow-connected and 3-Flow-connected}
\label{Z3-sec}

The main goal of this section is to prove Theorem~\ref{Z3-flow-main-thm}, characterizing graphs that are $3$-flow-connected and that are $\Z_3$-flow-connected; these two graph classes turn out to be identical.  This result follows immediately from 
Theorem~\ref{thm:mod-k-orientation}, by taking $k=1$.

\begin{theorem}\label{cor:3-flow}\label{Z3-flow-main-thm}
    If a graph $G$ admits a $3$-flow, 
    then the following statements are equivalent:
    \begin{itemize}
        \item $G$ is $3$-flow-connected.
        \item $G$ is $\Z_3$-flow-connected.
        \item $G$ is Eulerian and, for every two distinct vertices $w$ and $x$, $G$ contains at most $4$ 
		edge-disjoint $w,x$-paths.
    \end{itemize}
\end{theorem}

During our proof we develop tools and a framework that allow us to prove results of the same flavor that are more general. We begin by proving a key ``flip-equivalence'' lemma.
We use the term \emph{Eulerian} for both directed graphs, meaning that $d^+(v)=d^-(v)$ for all vertices $v$, and undirected graphs, meaning that every vertex has even degree. In each instance, our intended
meaning should be clear from context.
For a digraph $D_i$, we typically write $d^+_i(v)$ and $d^-_i(v)$, rather than the more cumbersome $d^+_{D_i}(v)$ and $d^-_{D_i}(v)$.

Note that each $\Z_3$-flow $f$ is the same as some $\Z_3$-flow $f'$ with $f'(e)=1$ for all $e\in E(G)$. Indeed, if $f(e)=2$ for some edge $e$, then we simply reverse the orientation of $e$ and change its flow value to $1$. For each directed cycle $C$ with $f(e)=1$ for all $e\in E(C)$, the only way that we can change flow values 
on $C$ to get another $\Z_3$-flow is to add $1$; equivalently, we can reverse all the edges of $C$. This motivates the following definition.

\begin{definition}
	Two orientations $D'$ and $D''$ of a graph are \emph{flip-adjacent} if $D'$ can be transformed
	to $D''$ by reversing the edges of a directed cycle. And two orientations $D'$ and $D''$ are 
	\emph{flip-equivalent} if there exists a sequence $D_0,\ldots,D_s$ with $D_0=D'$ and $D_s=D''$ such that $D_i$ and $D_{i-1}$ are flip-adjacent for each $i\in[s]$. 
\end{definition}

It is easy to check that flip-equivalence is indeed an equivalence relation. Thus, the $\Z_3$-flow-connectedness of a graph $G$ reduces to---in fact, is equivalent to---the flip equivalence of all its $\Z_3$-orientations (these are orientations $D$ of $G$ such that $d_D^+(v) - d_D^-(v) \equiv 0 \pmod{3},$ for each $v\in V(G)$, where $d_D^+(v)$ and $d_D^-(v)$ denote, respectively, the out-degree and in-degree of $v$).
Clearly, reversing the edges of a directed cycle preserves the values of $d^+(v)$ and $d^-(v)$ for all $v\in V(G)$. So, for two orientations $D_1$ and $D_2$ of $G$ to be flip-equivalent, we must have
$d^+_1(v)=d^+_2(v)$ and $d^-_1(v)=d^-_2(v)$ for all vertices $v$.  Our first lemma shows that this obvious necessary condition is also sufficient.

\begin{lem}[Flip-Equivalence Lemma]
	Two orientations $D_1$ and $D_2$ of a multigraph $G$ are flip-equivalent if and only if 
	$d_1^+(v)=d_2^+(v)$ for all $v\in V(G)$.
	\label{flip-equiv-lem}
\end{lem}
\begin{proof}
	Reversing the edges of a directed cycle maintains $d^+(v)$ for all $v\in V(G)$, so the
	hypothesis is obviously necessary.  Now we show that it is sufficient.

	We prove that if $D_1$ and $D_2$ are orientations of the same multigraph with the same out-degrees (so also the same in-degrees), then $D_1$ and $D_2$ are flip-equivalent.
	For a multigraph $G$ and an orientation $D$, let 
	$f(G,D):= \sum_{v\in V(G)}\max\{0,d_D^+(v)\times d_D^-(v)-1\}$.
	We use induction on $f(G,D_1)$; 
    first suppose that $f(G,D_1)=0$.
	Now we use induction on $|V(G)|$.
	Consider a component $C$ of $G$. If some vertex $v$ of $C$ has $\min\{d^+(v),d^-(v)\}=0$,
	then assume by symmetry that $d^-(v)=0$; that is, $v$ is a source.
	So all edges incident to $v$ are oriented outward in both $D_1$ and $D_2$. Thus, we can delete $v$ and proceed by induction. Assume instead that $\min\{d^+(v),d^-(v)\}\ge 1$ for all $v\in V(C)$. Since $f(G,D_1)=0$, we know that $d_{1}^+(v)=d_{1}^-(v)=1$ for all $v\in V(C)$.
	So $C$ is a directed cycle. If $D_1$ and $D_2$ differ on $C$ then since $d_2^+(v)=d_2^-(v)=1$ for all $v\in V(C)$, we simply reverse all edges of $C$. Now we delete $V(C)$ and proceed by induction.  This handles the base case.

	Now we prove the induction step. Since $f(G,D_1)\ge 1$, there exists $v$ with $d^+_1(v)\times d^-_1(v)-1 \ge 1$. By symmetry, we assume that $d_1^+(v)\ge d_1^-(v)\ge 1$. (Thus, also $d_2^+(v)\ge d_2^-(v)\ge 1$.) Consider an edge $e$ oriented into $v$ in $D_1$. We will find an edge $e'$ also incident to $v$ that is oriented 
	oppositely (into/out of $v$) of $e$ in both $D_1$ and $D_2$.  The number of edges
	oriented oppositely of $e$ in $D_1$ is precisely $d_1^+(v)$. And the number of edges oriented oppositely of $e$ in $D_2$ is at least $\min\{d_2^+(v),d_2^-(v)\} = d_2^-(v)$. So the number of edges incident with $v$ that are oriented oppositely from $e$ in either $D_1$ or $D_2$, with multiplicity, is at least $d_1^+(v)+d_2^-(v) = d_1^+(v)+d_1^-(v)=d(v)$. But note that $e$ is not
    among these. Thus, by Pigeonhole Principle some edge $e'$ incident to $v$ is oriented oppositely of $e$ in both $D_1$ and $D_2$.
	Now we split $e$ and $e'$ off of $v$; call the new vertex $v'$ and the resulting graph $G'$. Let $D_1'$ and $D_2'$ be the orientations of $G'$ inherited from $G$.  

	Note that each vertex in $V(G')$ has the same in-degree, and same out-degree, in both $D_1'$ and $D_2'$.  But also, $f(G',D_1') < f(G,D_1)$, so by the induction hypothesis orientations $D_1'$ and $D_2'$ are flip-equivalent on $G'$. (This flip equivalence in $G'$ is witnessed by a sequence $C_1',\ldots,C_s'$ such that starting from $D_1'$ and successively reversing the edges of $C_1',\ldots,C_s'$ yields $D_2'$.)  
    We want to lift the flips in $G'$ (on the cycles $C_i'$) to flips in $G$; however, it is possible that a cycle in $G'$ might visit both $v$ and $v'$.  This subgraph lifts in $G$ not to a cycle, but to a connected subgraph with each vertex of degree $2$, except for $v$ which has degree $4$. In fact, the corresponding subdigraph has $d^+(w)=d^-(w)=1$ for all vertices $w$, except that $d^+(v)=d^-(v)=2$. But we can decompose this digraph into $2$ directed cycles, and reverse each of them.  This allows us to lift the flips from $G'$ to $G$, as desired, and thus proves the lemma.
\end{proof}

The following lemmas are well-known, but for completeness we include short proofs.

\begin{lem}[Flow Decomposition Theorem]\label{lem:FlowDecomposition}
	If $D$ is a digraph with $s,t\in V(D)$ such that $d^+_D(v) = d^-_D(v)$ for all $v\in V(D)\setminus\{s,t\}$ and $d^+_D(s)\ge d^-_D(s)$, then $D$ can be decomposed into
	directed cycles and $d^+_D(s)-d^-_D(s)$ directed $s,t$-paths.
\end{lem}
\begin{proof}
	Let $k:=d^+_D(s)-d^-_D(s)$, and form $D'$ from $D$ by adding $k$ copies of the directed {arc $(t,s)$}.
	Since $D'$ is Eulerian, it can be decomposed into directed cycles. (The proof is by induction
	on $|E(D')|$; we simply remove a directed cycle and apply the induction hypothesis.)  Now each of
	the $k$ cycles in the decomposition that contains one of the added arcs $(t,s)$ corresponds to a directed $s,t$-path in $D$.
\end{proof}

\begin{lem}[$k$-flow/$\Z_k$-flow Equivalence Theorem~\cite{tutte54}]
	\label{Zk-k-equiv-lem}
	Fix $k\in \Z^+$.  If a graph $G$ has a $\Z_k$-flow $f$, then $G$ also has a $k$-flow 
	$g$ such that $g(e)\equiv f(e)\pmod k$ for each $e\in E(G)$.
\end{lem}
\begin{proof}
	We assume that $f(e)\in\{1,\ldots,k-1\}$ for each edge $e$. Since $f$ is a $\Z_k$-flow, we have $\sum_{w\in N^+(v)}f(vw) - \sum_{u\in N^-(v)}f(uv) \equiv 0 \pmod k$ for each $v\in V(G)$.
	If $f$ is not also a $k$-flow, then there exists $x\in V(G)$ such that $\sum_{w\in N^+(x)}f(xw) - \sum_{u\in N^-(x)}f(ux) >0$; we call the ``net flow'' at $x$ positive.  
	Following out-edges with positive flow value (and possibly in-edges with negative flow value), we can find a path $P$ from $x$ to a vertex $y$ with net flow negative.  By subtracting $k$ from the flow value of each edge of $P$ (or adding $k$ to in-edges with negative flow), we can decrease the sum, over all vertices of $G$, of the absolute values of the net flows; note that the net flow only changes at $x$ and $y$. By repeating this process, we reach a flow where the sum is $0$, which is thus a $k$-flow.
\end{proof}

For each positive integer $k$, a \emph{$\Z_{2k+1}$-orientation} is an orientation $D$ of $G$ such that 
$d_D^+(v) - d_D^-(v) \equiv 0 \pmod{2k+1}$
for every vertex $v\in V(G)$.
Here $d_D^+(v)$ and $d_D^-(v)$ denote the out-degree and in-degree of $v$.
The \emph{$\Z_{2k+1}$-orientation reconfiguration graph} $\O(G,\Z_{2k+1})$ has  as its vertices all $\Z_{2k+1}$-orientations on $G$, and two orientations $D_1,D_2$ are adjacent if $D_2$ can be obtained from $D_1$ by reversing the edges of a directed cycle.

We are now ready to characterize $\Z_3$-flow-connected graphs.  
Before giving the proof, we make a simple observation, and state a useful lemma.

\begin{obs}\label{Z3-obs}
	For any graph $G$, the $\Z_3$-orientation reconfiguration graph $\O(G,\Z_3)$ is isomorphic to the $\Z_3$-flow reconfiguration graph $\F(G, \Z_3)$. 
	Thus, a graph $G$ is $\Z_3$-flow-connected if and only if $G$ is $\Z_3$-orientation-connected.
\end{obs}
\begin{proof}
Each $\Z_3$-flow is the same as a $\Z_3$-orientation by reversing each edge with flow value $2$. 
And each $\Z_3$-orientation can be viewed as a $\Z_3$-flow (with flow value $1$ everywhere). 
\end{proof}

\begin{lem}[\cite{YOH-MLV}]\label{bip-paths-lem}
If $G$ is a $d$-regular bipartite graph (possibly with multiple edges), then for each $w\in V(G)$ there exists a vertex $x\in V(G)$ in the other part of the bipartition such that $G$ contains $d$ edge-disjoint $w,x$-paths.
\end{lem}

The main motivation for our next result is the case $k=1$.  But the proof works equally well when $k$
is any positive integer.  So we present it in full generality.

\begin{theorem}\label{mod-ori-conn-thm}
Let $k$ be a positive integer and let $G$ be a graph that admits a $\Z_{2k+1}$-orientation. Now $G$ is $\Z_{2k+1}$-orientation-connected if and only if $G$ is Eulerian and has no pair of vertices $w,x$ such that $G$ contains $4k+2$ edge-disjoint $w,x$-paths.
\end{theorem}

\begin{proof}
We first prove necessity. Assume $G$ is $\Z_{2k+1}$-orientation-connected.
Suppose there exists $v \in V(G)$ with odd degree. Let $D$ be a $\Z_{2k+1}$-orientation of $G$, and form $D^R$ from $D$ by reversing every edge. So $D^R$ is also a $\Z_{2k+1}$-orientation. 
Moreover, $d_{D^R}^+(v)=d_D^-(v)$ and $d_D^-(v)\neq d_D^+(v)$ because $d(v)$ is odd. By Lemma~\ref{flip-equiv-lem}, orientations $D$ and $D^R$ are not flip-equivalent, contradicting the assumption. Hence, $G$ is Eulerian.

Now suppose that $G$ is Eulerian and $G$ has vertices $w,x \in V(G)$ such that $G$ has $4k+2$ edge-disjoint $w,x$-paths. Let $H$ be the union of these paths. Since $G$ is Eulerian, the graph $G - E(H)$ is also Eulerian; let $D_0$ be an Eulerian orientation of $G - E(H)$. Orient all paths in $H$ from $w$ to $x$, and denote the resulting orientation by $D_H$. 
Then 
$d_{D_H}^+(w)=4k+2=d_{D_H}^-(x)$ and
$d_{D_H}^-(w)=0=d_{D_H}^+(x)$, and for every other vertex $v \in V(H)$ we have $d_{D_H}^+(v)=d_{D_H}^-(v)$. Let $D_1 := D_0 \cup D_H$ and let $D_2 := D_0 \cup D_H^R$. 
Note that $D_1$ and $D_2$ are both $\Z_{2k+1}$-orientations of $G$. However, $d_{D_1}^+(w)=4k+2 > 0=d_{D_2}^+(w)$, and hence $D_1$ and $D_2$ are not flip-equivalent by Lemma~\ref{flip-equiv-lem}, a contradiction.

We now prove the sufficiency. Assume that $G$ is Eulerian and that $G$ has no such pair of vertices $w,x$ as above. Let $D_1$ be an Eulerian orientation of $G$. Suppose, for a contradiction, that some $\Z_{2k+1}$-orientation $D_2$ of $G$ is not flip-equivalent to $D_1$. By Lemma~\ref{flip-equiv-lem}, there exists a vertex $v$ such that $d_{1}^+(v)\neq d_{2}^+(v)$. Since $D_1$ is Eulerian, we have $d_{1}^+(v)=d_{1}^-(v)$, and hence $d_{2}^+(v)\neq d_{2}^-(v)$.
Let $U:=\{u\in V(G): d_{2}^+(u)>d_{2}^-(u)\}$ and let $W:=\{w\in V(G): d_{2}^-(w)>d_{2}^+(w)\}$. 
Note that $U\neq \emptyset$ and $W\neq \emptyset$.

Since $D_2$ is a $\Z_{2k+1}$-orientation, for every $v\in V(G)$ we have $d_{2}^+(v)-d_{2}^-(v)\equiv 0 \pmod{2k+1}$. Moreover, since $G$ is Eulerian, $d_G(v)$ is even for all $v$, and hence $d_{2}^+(v)-d_{2}^-(v)\equiv d_G(v)\equiv 0 \pmod{2}$. It follows that $d_{2}^+(v)-d_{2}^-(v)\equiv 0 \pmod{4k+2}$. Therefore, for each $v\in V(G)$ we have $d_{2}^+(v)-d_{2}^-(v)=(4p+2)t$ for some $t\in \mathbb{Z}$, 
and in particular $t\neq 0$ for each $v\in U\cup W$.

By the Flow Decomposition Theorem (Lemma~\ref{lem:FlowDecomposition}), the directed edges of $D_2$ can be decomposed into a disjoint union of directed $U,W$-paths, which we call $\mathcal{P}$, and directed cycles such that each $u\in U$ is the source of exactly $d_{2}^+(u)-d_{2}^-(u)$ paths and each $w\in W$ is the sink of exactly $d_{2}^-(w)-d_{2}^+(w)$ paths.

We construct an auxiliary bipartite multigraph $B_0$ with bipartition $(U,W)$, where each path in $\mathcal{P}$ 
from $u$ to $w$ corresponds to an edge $uw$. Now $d_{B_0}(x)=|d_{D_2}^+(x)-d_{D_2}^-(x)|$ for each $x\in U\cup W$; hence, each vertex of $B_0$ has degree divisible by $4k+2$.  Form $B$ from $B_0$ by splitting each $v$ with $d_{B_0}(v)>4k+2$ into $d_{B_0}(v)/(4k+2)$ vertices, each of degree $4k+2$.
Note that $B$ is bipartite and $(4k+2)$-regular. By Lemma~\ref{bip-paths-lem}, there exist vertices $w,x\in V(B)$ such that $B$ contains $4k+2$ edge-disjoint $w,x$-paths. These paths correspond to $4k+2$ edge-disjoint $w,x$ walks in $G$, which can be reduced to $4k+2$ edge-disjoint $w,x$-paths in $G$; this contradicts our hypothesis. 
\end{proof}

When $k=1$, the previous result characterizes graphs that are $\Z_3$-flow-connected. What remains to do in this section is to show that precisely the same graphs are $3$-flow-connected. Again, our arguments work in more generality. We recall a few types of flows, and for each we define a corresponding reconfiguration graph.

Fix integers $p,q$ with $p \ge 2q>0$. A \emph{circular $(p,q)$-flow} on a graph $G$ is a flow  $f:E(G)\to \pm\{q,q+1,\dots,p-q\}$. 
A \emph{modular $(p,q)$-flow} is a flow $f:E(G)\to \{q,q+1,\dots,p-q\}\subseteq \mathbb{Z}_p$.  

\begin{itemize}
\item The \emph{circular $(p,q)$-flow reconfiguration graph} $\mathcal{F}_{\mathrm{cir}}(G;p,q)$ has all circular $(p,q)$-flows as vertices, and two flows $f_1,f_2$ (with the same underlying orientation) are adjacent if their difference $f_1 - f_2$ is supported on a cycle of $G$.  

\item The \emph{modular $(p,q)$-flow reconfiguration graph} $\mathcal{F}_{\mathrm{mod}}(G;p,q)$ is defined analogously for modular $(p,q)$-flows.
\end{itemize}

Each circular $(k,1)$-flow has flow values $\{\pm 1,\pm 2, \ldots, \pm (k-1)\}$; so, it is precisely a $k$-flow. Using the characterization in Theorem~\ref{mod-ori-conn-thm}, we get the following equivalence.

\begin{theorem}
\label{thm:mod-k-orientation}
Let $k$ be a positive integer and let $G$ be a graph that admits a circular $(2k+1,k)$-flow. The following $4$ statements are equivalent:
\begin{enumerate}[label=(\roman*)]
    \item\label{Equiv1} The circular $(2k+1,k)$-flow reconfiguration graph $\mathcal{F}_{\mathrm{cir}}(G;2k+1,k)$ is connected.
    \item\label{Equiv2} The modular $(2k+1,k)$-flow reconfiguration graph $\mathcal{F}_{\mathrm{mod}}(G;2k+1,k)$ is connected.
    \item\label{Equiv3} The $\Z_{2k+1}$-orientation reconfiguration graph $\O(G,\Z_{2k+1})$ is connected.
    \item\label{Equiv4} $G$ is Eulerian and, for each pair $w,x$ of distinct vertices, $G$ contains at most $4k$ edge-disjoint $w,x$-paths.
\end{enumerate}
\end{theorem}

\begin{proof}
The equivalence between \ref{Equiv3} and \ref{Equiv4} is exactly Theorem~\ref{mod-ori-conn-thm}. (By Observation~\ref{eulerian-obs},
no Eulerian graph has an odd edge-cut. Thus, by Menger's Theorem, the maximum number of edge-disjoint paths between every $2$ vertices is even.) And the equivalence between \ref{Equiv2} and \ref{Equiv3} is easy, so we just sketch the proof. In a modular $(2k+1,k)$-flow, each flow value is $k$ or $k+1$; multiplying each flow value by $2$ gives a flow with each value in $\{-1,1\}$, and reversing each edge with flow value $-1$ gives a $\Z_{2k+1}$-orientation. But this transformation also works in reverse: starting from a $\Z_{2k+1}$-orientation and giving each edge flow value $k$ yields a modular $(2k+1,k)$-flow. It is easy to check that a walk between 2 flows of one type ``lifts'' to a walk between 2 flows of the other.

We now show that \ref{Equiv1} implies \ref{Equiv2}. Assume that $G$ is circular $(2k+1,k)$-flow-connected. Let \(\rho\) be a mapping from \(V(\mathcal{F}_{\mathrm{cir}}(G;2k+1,k))\) to $V(\mathcal{F}_{\mathrm{mod}}(G;2k+1,k))$ such that \(\rho(f)(e)\equiv f(e)\pmod{2k+1}\) for each \(e\in E(G)\).
Note that \(\rho\) is well-defined. By Lemma~\ref{Zk-k-equiv-lem}, $\rho$ is surjective. 
For two adjacent flows $(D,f_1)$, $(D,f_2)\in V(\mathcal{F}_{\mathrm{cir}}(G;2k+1,k))$, one can see that either $\rho(f_1)=\rho(f_2)$, or $\text{supp}(\rho(f_1)-\rho(f_2))$ is a cycle. This means, either they are mapped into one vertex, or their images are adjacent in $\mathcal{F}_{\mathrm{mod}}(G;2k+1,k)$. Since \(\rho\) is surjective, our proof is complete. 

It remains to prove that \ref{Equiv2} implies \ref{Equiv1}. Assume that $G$ is modular $(2k+1,k)$-flow-connected. By the equivalence of \ref{Equiv2}, \ref{Equiv3}, and \ref{Equiv4}, it follows that $G$ is $\Z_{2k+1}$-orientation-connected and $G$ is Eulerian. 

Fix an Eulerian orientation $D_0$ of $G$, and let $f_0$ be the circular $(2k+1,k)$-flow formed from 
$D_0$ by assigning every edge flow value $k$. Let $(D,f)$ be an arbitrary circular $(2k+1,k)$-flow on $G$. We show that $(D,f)$ is circular-$(2k+1,k)$-flow-equivalent 
to $(D_0,f_0)$. By reversing every edge with negative value, we assume that $f(e)\in \{k,k+1\}$ for every edge $e$. Let $G_k$ and $G_{k+1}$ be the spanning subgraphs of $G$ consisting of edges in $(D,f)$, respectively, with values $k$ and $k+1$. If for every vertex $v$ we have $d_{G_k}^+(v)-d_{G_k}^-(v)=0$, then it follows that $d_{G_{k+1}}^+(v)-d_{G_{k+1}}^-(v)=0$, and hence both $G_k$ and $G_{k+1}$ are Eulerian. We decompose $G_{k+1}$ into directed cycles and subtract 1 from the values on each cycle. This gives a flow $(D,f')$ such that $D$ is an Eulerian orientation and $f'(e)=k$ for all $e$. By Lemma~\ref{flip-equiv-lem}, the flows $(D,f')$ and $(D_0,f_0)$ are equivalent, and we are done.

So assume instead that there exists a vertex $v$ such that $d_{G_k}^+(v)-d_{G_k}^-(v) > 0$; since $v$ has net flow 0, consequently $d_{G_{k+1}}^+(v)-d_{G_{k+1}}^-(v) < 0$. Now reverse every edge $e$ in $G_{k+1}$, assign the flow value $-(k+1)$ to $e$, and denote the resulting oriantation and flow by $(D_1,f_1)$. Note that, since $-(k+1)\equiv k \pmod{2k+1}$, for every edge $e$ we have that $f_1(e)\equiv k\pmod{2k+1}$.
Thus $D_1$ is a $\Z_{2k+1}$-orientation of $G$; this is because $k$ and $2k+1$ are relatively prime. But note that $d_{D_1}^+(v)-d_{D_1}^-(v) = (d_{G_k}^+(v)-d_{G_k}^-(v)) - (d_{G_{k+1}}^+(v)-d_{G_{k+1}}^-(v)) > 0$. 
Because \ref{Equiv2} is equivalent to \ref{Equiv3}, we know that $\O(G,\Z_{2k+1})$ is connected. And because \ref{Equiv3} is equivalent to \ref{Equiv4}, we know that $G$ has an Eulerian orientation $D_0$, which is also a $\Z_{2k+1}$-orientation. But this contradicts the Flip-Equivalence Lemma (Lemma~\ref{flip-equiv-lem}), since $d^+_1(v)>d^+_0(v)$.
\end{proof}

\section{\texorpdfstring{$\bm{\Z_4}$}{Z4}-Flow-connectedness and \texorpdfstring{$\bm{4}$}{4}-Flow-connectedness in Cubic Graphs}
\label{Z4flow-sec}

In this section, we prove that every cubic bipartite graph is $4$-flow-connected (and hence $\Z_4$-flow-connected). We also show that the prism graphs $C_n\square P_2$ ($n\ge 4$) are $\Z_4$-flow-connected. In contrast, we give sufficient conditions for a cubic graph to fail to be $\Z_4$-flow-connected. 

Throughout this section, we only consider cubic graphs that are $3$-edge-colorable; otherwise, no nowhere-zero $4$-flow exists and the graph's $4$-flow-connectedness is trivial. We call a $k$-flow $(D,f)$ \emph{positive} if every edge $e$ satisfies $f(e)>0$. (Noting that every $k$-flow is the same as a positive $k$-flow, we may regard the vertices of $\F(G,4)$ as precisely the distinct positive $4$-flows of $G$.) We begin with the following lemma. 

\begin{lem}\label{orient-M-lem}
Let $G$ be a cubic bipartite graph with bipartition $(A,B)$. Let $(D,f)$ be a positive
$4$-flow of $G$, and let $M$ be the $1$-factor consisting of the edges with flow value $2$ in $f$. Then $(D,f)$ is flow-equivalent to a positive $4$-flow in which every edge of $M$ still has value $2$ and is oriented from $B$ to $A$.
\end{lem}
\begin{proof}
	Suppose an edge $e$ of $M$ is oriented from $A$ to $B$ under $D$.  
    We now grow a directed path $P$ starting from $e$. 
    Let $x$ be the current head of $P$ at each step.
	If $x\in A$ and $x$ has an outgoing edge of value $2$, then we follow such an edge. 
	If $x\in B$ and $x$ has an outgoing edge of value $1$, then we follow it.   
	Otherwise, we follow an outgoing edge of value $3$.
    
    To begin, we show that the path can always be extended. 
	If $x\in A$, then the path enters $x$ either along an edge of value $3$, in which case $x$ has an outgoing edge of value $2$, or along an edge of value $1$, in which case $x$ has an outgoing edge of value $2$ or $3$. Similarly, if $x\in B$, then the path enters $x$ either along an edge of value $3$, in which case $x$ has an outgoing edge of value $1$, or along an edge of value $2$, in which case $x$ has an outgoing edge of value $1$ or $3$.
	Thus, we can continue to grow this path until we return to some vertex that we have already visited; at this point, we say that we close the cycle, which we denote by $C$.  
    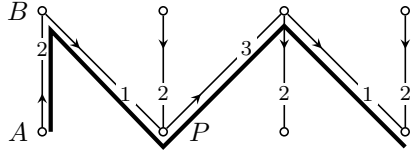
\begin{figure}[!t]
\begin{tikzpicture}[semithick,
    scale=.8, xscale=2, 
    vertex/.style={
        circle, 
        draw, 
        fill=white, 
        minimum size=3pt,  
        inner sep=0pt
    },
    directed edge/.style={
        decoration={
            markings,
            mark=at position 0.3 with {\arrow{stealth}}
        },
        postaction={decorate}
    },
    weight label/.style={
        draw=none, 
        fill=white, 
        inner sep=1.5pt,   
        font=\small,       
        pos=0.7
    }
]

\draw (-2.55,0) node[draw=none, fill=none] {};
\draw (.8,0) node[draw = none] {\footnotesize{$A$}};
\draw (.8,2) node[draw = none] {\footnotesize{$B$}};
\draw (2.3,0) node[draw = none] {\footnotesize{$P$}};
    \foreach \i in {1,...,4} {
        \node[vertex] (a\i) at (\i, 0) {};
    }
    
    \foreach \i in {1,...,4} {
        \node[vertex] (b\i) at (\i, 2) {};
    }

    \draw[directed edge] (a1) -- node[weight label] {\scriptsize{2}} (b1);
    
    \draw[directed edge] (b1) -- node[weight label] {\scriptsize{1}} (a2);
    
    \draw[directed edge] (b2) -- node[weight label] {\scriptsize{2}} (a2);
    
    \draw[directed edge] (a2) -- node[weight label] {\scriptsize{3}} (b3);
    
    \draw[directed edge] (b3) -- node[weight label] {\scriptsize{2}} (a3);
    
    \draw[directed edge] (b3) -- node[weight label] {\scriptsize{1}} (a4);
    
    \draw[directed edge] (b4) -- node[weight label] {\scriptsize{2}} (a4);

    \draw[ultra thick] (1.07,0) -- (1.07,1.68) -- (2,-.25) -- (3,1.75) -- (4,-.25);
\end{tikzpicture}
	\caption{Extending the path $P$.%
}
\end{figure}

    We will ensure that $C$ contains at least one edge with flow value $2$.
    If the closing edge has value $2$, then clearly $C$ contains an edge with value $2$. 
    If the closing edge has value $1$, then its head is in $A$, say at a vertex $v$.  Now $v$ has an incident edge $e$ with value $2$. 
    If we previously entered and exited $v$, then we used $e$
    to do one of these. But, by design, we must have followed $e$ from $A$ to $B$, so $e$ lies on $C$.
	Finally, note that we actually cannot close $C$ on an edge with value $3$, unless we close it at the start of $P$.  The reason is that the vertex $w$ at the head of the edge with value $3$ must have two outgoing edges.  So we could not
	have previously entered and exited $w$.  
    Thus, $C$ contains an edge with value $2$.

	Recall that every edge of $C$ with flow value $2$ is directed from $A$ to $B$.
    Now we subtract $4$ from the flow value of each edge on $C$, which reverses the orientations of all of its edges, interchanges edges with flow values $1$ and $3$, and maintains those with flow value $2$. This results in more edges with flow value $2$ oriented from $B$ to $A$. By repeating this process, we eventually get all edges of $M$ oriented from $B$ to $A$,
	as desired.
\end{proof}

\begin{theorem}\label{thm:4-flow-bipartite-cubic}
Every cubic bipartite graph is $4$-flow-connected.
\end{theorem}

\begin{proof}
Let $G$ be a cubic bipartite graph with bipartition $(A,B)$. Let $f_1$ and $f_2$ be two arbitrary positive $4$-flows on $G$. For each $i\in[2]$, let $M_i$ be the $1$-factor formed by the edges with flow value $2$ in $f_i$. To see the $4$-flow-equivalence between $f_1$ and $f_2$, by Lemma~\ref{orient-M-lem}, we may assume that each edge of $M_i$ is oriented from $B$ to $A$.  

We next transform each $f_i$ into a flow such that every edge of $E(G)\setminus M_i$ has flow value $1$ and is oriented from $A$ to $B$ (while preserving the property that every edge of each $M_i$ has flow value $2$ and is oriented from $B$ to $A$).
Suppose instead that some edge $e_1\in E(G)\setminus M_1$ has flow value 3. So $e_1$ is oriented from $A$ to $B$; let $v$ be its head. Thus, there is an outgoing edge $e_2$ from $v$ (from $B$ to $A$) with flow value $1$.
Repeating the same argument at the head of each newly chosen edge, we eventually obtain a directed cycle whose edge values alternate between $3$ and $1$. Now we subtract 2 from each cycle edge, which converts each cycle edge to have flow value 1, and be oriented from $A$ to $B$. Starting from $f_1$, we repeat this process until every edge outside of $M'$ has flow value 1, directed from $A$ to $B$.

Finally, we consider a cycle $C$ in the symmetric difference $M_1\triangle M_2$.  Starting from $f_1$ with its orientation, we subtract 3 from the flow value for each edge of $C$.  This changes the edges of $M_1\cap E(C)$ to have flow value 1, directed from $A$ to $B$, and changes the edges of $M_2\cap E(C)$ to have flow value 2, directed from $B$ to $A$. By repeating this step for each cycle of $M'\triangle M''$, we transform $f_1$ into $f_2$, which finishes the proof.
\end{proof}

It follows directly from the first sentence after Conjecture~\ref{conjA} that every cubic bipartite graph is $\mathbb Z_4$-flow-connected. Next, we present a sufficient condition under which a cubic graph is not $\mathbb{Z}_4$-flow-connected.

\begin{theorem}\label{notZ4flow-connectedII}
Let $G$ be a $3$-edge-colorable cubic graph. If $G$ contains a $3$-cycle, then $G$ is not $\Z_4$-flow-connected, and hence not $4$-flow-connected.
\end{theorem}

\begin{proof}
Since $G$ is $3$-edge-colorable, we choose a proper $3$-edge-coloring of $G$, and let $M \subseteq E(G)$ be one of its color classes. 
Define a $\mathbb Z_4$-flow $(D,f)$ by assigning flow value $2$ to the edges of $M$ and flow value $1$ to all remaining edges, with a suitable orientation. Let $vxy$ be a triangle in $G$. 
Because the three edges of the triangle receive distinct colors, exactly one of them has flow value $2$.  
Without loss of generality, assume that $f(xy)=2$ and $f(vx)=f(vy)=1$. 
So the two edges $vx$ and $vy$ are oriented either both away from $v$ or both toward $v$, 
depending on the orientation of $D$.

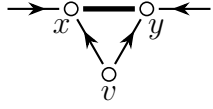
\begin{figure}[!h]
\centering
\begin{tikzpicture}[thick, scale=1]
\tikzstyle{uStyle}=[shape = circle, minimum size = 5pt, inner sep = 0pt,
outer sep = 0pt, draw, fill=white, semithick]
\tikzstyle{bStyle}=[shape = circle, minimum size = 5pt, inner sep = 0pt,
outer sep = 0pt, draw, fill=gray, semithick]
\tikzstyle{eStyle}=[shape = circle, minimum size = 5pt, inner sep = 0pt,
outer sep = 0pt, draw=none, fill=none, semithick]
\tikzset{every node/.style=uStyle}

\draw (0,0) node[eStyle] (w1) {} (1,0) node (x) {} (2,0) node (y) {} (3,0) node[eStyle] (w2) {}
(1.5,-.866) node (v) {};

\draw[shorten <= 2pt, shorten >= 2pt, line width=2.0pt] (x) -- (y);
\draw[shorten <= 2pt, shorten >= 2pt, mid arrow] (v) -- (x);
\draw[shorten <= 2pt, shorten >= 2pt, mid arrow] (v) -- (y);
\draw[shorten <= 2pt, shorten >= 2pt, mid arrow] (w1) -- (x);
\draw[shorten <= 2pt, shorten >= 2pt, mid arrow] (w2) -- (y);
\draw (v) ++ (0,-.25) node[eStyle] {$v$};
\draw (x) ++ (-.125,-.225) node[eStyle] {$x$};
\draw (y) ++ (.125,-.265) node[eStyle] {$y$};

\end{tikzpicture}
	\captionsetup{width=.72\textwidth}
	\caption{Each reconfiguration step maintains the invariant that arcs $vx$ and $vy$ each have flow value 1
and edge $xy$ has flow value 2.}
\end{figure}

We first show that the value on $xy$ cannot be changed in any valid reconfiguration step. 
Suppose, for contradiction, that some cycle $C_0$ (supporting the difference between two adjacent flows) contains $xy$. Since $C_0$ is a cycle, it must contain exactly one other edge incident with $x$ and one other edge incident with $y$.  With respect to the Eulerian orientation of $C$, these two edges have flow values $1$ and $3$. 
Thus $C_0$ contains the three distinct values $1,2,3$, and no non-zero element of $\mathbb Z_4$ can be added along $C_0$ without creating a zero value. Hence $C_0$ cannot be used in any valid reconfiguration step. 
Moreover, after any such step, the values on $vx$ and $vy$ remain $1$ and are either both directed out of $v$ or both into $v$. 
Thus the same obstruction persists, and $xy$ can never belong to any such cycle. 
Therefore $f(xy)=2$ is invariant throughout the reconfiguration process.

Now choose another color class $M'$ of the $3$-edge-coloring such that $xy \notin M'$. 
Analogously, $M'$ induces a $\mathbb Z_4$-flow $f'$ in which the edges of $M'$ have value $2$ and all remaining edges have value $1$. 
By construction, $f'(xy) = 1 \neq 2 = f(xy)$. 
Since the value $f(xy)=2$ is invariant under reconfiguration steps, $f$ and $f'$ lie in different components of $\F(G,\mathbb Z_4)$. 
\end{proof}

To conclude this section, we give an explicit family of non-bipartite cubic graphs $\G$ such that each $G\in \G$ is $\Z_4$-flow-connected.

\begin{theorem}
For any integer $n\ge 4$, the prism graph $C_n \square P_2$ is $\mathbb{Z}_4$-flow-connected.
\label{odd-prism-thm}
\end{theorem}

\begin{proof}
Let $G:=C_n \square P_2$.
Note that $G$ is cubic, and if $n$ is even, then $G$ is also bipartite. Hence, $G$ is $\mathbb Z_4$-flow-connected by Theorem~\ref{thm:4-flow-bipartite-cubic}. Thus, we assume that $n$ is odd.
Denote $V(G)$ by $\{v_1,\ldots,v_n\}\cup\{w_1,\ldots w_n\}$, where $v_iv_{i+1}$ and $w_iw_{i+1}$ (indices taken modulo $n$) form, respectively, the inner and outer cycles $C_{\rm in}$ and $C_{\rm out}$. 
Let $S:=\{v_iw_i:i\in[n]\}$. We call elements of $S$\emph{spokes}.

A perfect matching $M$ of $G$ is called \emph{admissible} if every component of $G-M$ is an even cycle. In any nowhere-zero $\mathbb Z_4$-flow on $G$, the edges of value $2$ form a perfect matching $M$; let $F:=G-M$, and note that $F$ is a 2-factor. For each $\Z_4$-flow $f$ on $G$, note that $M$ cannot consist of all $n$ spokes (if so, then $V(C_{\rm in})$ would have non-zero net flow in, a contradiction). Thus, we conclude that $F$ consists only of even cycles; this holds because the only vertex disjoint odd cycles in $G$ are $C_{\rm in}$ and $C_{\rm out}$. That is, every $M$ arising from a $\Z_4$-flow on $G$ is indeed admissible.

\begin{figure}[!h]
\centering
\begin{tikzpicture}[thick, scale=1]
\tikzstyle{uStyle}=[shape = circle, minimum size = 5pt, inner sep = 0pt,
outer sep = 0pt, draw, fill=white, semithick]
\tikzstyle{bStyle}=[shape = circle, minimum size = 5pt, inner sep = 0pt,
outer sep = 0pt, draw, fill=gray, semithick]
\tikzstyle{eStyle}=[shape = circle, minimum size = 5pt, inner sep = 0pt,
outer sep = 0pt, draw=none, fill=none, semithick]
\tikzset{every node/.style=uStyle}

\foreach \i in {1,2,...,7}
{
	\draw (\i,0) node (x\i) {} (\i,1) node (y\i) {};
	\draw[shorten >=3pt, shorten <= 3pt, dash pattern=on 2pt off 1pt on 2pt off 1pt] (x\i) -- (y\i);
}

\draw (5.5,0) node[eStyle] {$\cdots$};
\draw (5.5,1) node[eStyle] {$\cdots$};

\foreach \i/\j in {1/2,2/3,3/4,4/5,6/7}
{
	\draw[shorten <=3pt, shorten >=3pt] (x\i) -- (x\j); 
	\draw[shorten <=3pt, shorten >=3pt] (y\i) -- (y\j);
}

\draw[line width=1.5pt] (1.1,.15) -- (1.1,.85);
\draw[line width=1.5pt] (2.1,.15) -- (2.1,.85);
\draw[line width=1.4pt] (3.15,.1) -- (3.85,.1);
\draw[line width=1.4pt] (3.15,.9) -- (3.85,.9);
\draw[line width=1.5pt] (5.1,.15) -- (5.1,.85);
\draw[line width=1.4pt] (6.15,.1) -- (6.85,.1);
\draw[line width=1.4pt] (6.15,.9) -- (6.85,.9);
\draw[shorten <=2pt, shorten >=2pt] (y1) to[in=25, out=155, looseness=0.6] (y7);
\draw[shorten <=2pt, shorten >=2pt] (x1) to[in=-25, out=-155, looseness=0.6] (x7);

\draw (4.5,1.2) node[eStyle] {\scriptsize{$C_{in}$}};
\draw (4.5,-.2) node[eStyle] {\scriptsize{$C_{out}$}};
\end{tikzpicture}
	\captionsetup{width=.625\textwidth}
	\caption{A prism graph $C_n\square P_2$, with each spoke drawn as a dashed line, along with a perfect matching $M$ (drawn in bold). Here $k_M=2$, since $M$ contains 2 pairs of cycle edges.}
\end{figure}
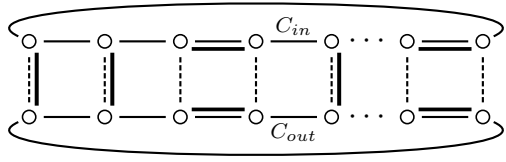

Since $n$ is odd, every perfect matching contains an odd number of spokes. And if $M$ is admissible,
then $M$ contains at least one pair of parallel cycle edges $v_iv_{i+1}$ and $w_iw_{i+1}$; we call these parallel edges a \emph{cycle pair}. Let $k_M$ denote the number of cycle pairs in $M$.

We now describe the flow operations. For a fixed admissible matching $M$, assign value $2$ to edges in $M$ and value $1$ to all other edges (with a suitable orientation). On each cycle of $F$, vertices alternate between Type~1 (both incident edges directed out) and Type~2 (both incident edges directed in). Subtracting $2$ along the edges of a cycle of $F$ switches these two types. We call this operation \emph{type-switching}. Hence, all flows associated with a fixed admissible matching are $\mathbb Z_4$-flow-equivalent.

We next describe our key local move. Let $K_i:=v_iv_{i+1}w_{i+1}w_i$. 
A \emph{$C_4$-swap} on $K_i$ means adding $1$ to the flow value of each edge of $K_i$.  
This operation replaces the cycle pair with 
the spoke pair $(v_iw_i, v_{i+1}w_{i+1})$ and decreases $k_M$ by one. 
If $(v_iv_{i+1}, w_iw_{i+1})$ is a cycle pair, and $k_M\ge 2$, then by applying type-switching on components of $F$ we may assume that the orientation on $K_i$ allows a $C_4$-swap. 
Conversely, a spoke pair can be transformed into a cycle pair by the same operation.

Starting from any admissible matching, we use type-switching on components of $F$ to prepare a 4-cycle containing a cycle pair so that we can apply a $C_4$-swap. Repeating this process reduces $k_M$ until we obtain an admissible matching with $k_M=1$.

When $k_M=1$, the unique cycle pair determines a Hamilton cycle in $F$. Applying a $C_4$-swap on a suitable 4-cycle transforms a spoke pair into a cycle pair, yielding an admissible matching with $k_M=2$; and the two cycle pairs can be chosen to be consecutive.

Finally, for admissible matchings with $k_M=2$ and consecutive cycle pairs, type-switching allows us to normalize orientations so that repeated $C_4$-swaps shift cycle pairs along the prism. Hence, all such configurations are $\mathbb Z_4$-flow-equivalent.
\end{proof}

\section{Flow Connectedness of Eulerian Graphs}\label{sec:Eulerian}

Eulerian graphs are precisely the graphs that must have an $A$-flow whenever $|A|\ge 2$ (and a $k$-flow whenever $k\ge 2$). In Section~\ref{Z3-sec}, we showed that the analogous statement fails for both $\mathbb{Z}_3$-flow-connectedness and $3$-flow-connectedness. But we believe that something similar is true.  In this short section, we pose a conjecture in this direction. We also offer some evidence to support this conjecture.\footnote{With the goal of presenting more accessible results first, we put this section at this point in the paper even though a few of its proofs rely on results from later sections (namely, that $C_2$, $C_3$, and $C_4$ are ``contractible'' for $A$-flow-connectedness when $|A|\ge 5$, and that all $4$-edge-connected graphs are $A$-flow-connected when $|A|$ is large).}

\begin{conj}\label{eulerian-conj}
Let $G$ be an Eulerian graph. The following two claims hold.
    \begin{itemize}
        \item $G$ is $k$-flow-connected 
    whenever $k\ge 4$.
        \item $G$ is $A$-flow-connected 
    whenever $A$ is an abelian group with $|A|\ge 4$.
    \end{itemize}
\end{conj}

It turns out that proving Conjecture~\ref{eulerian-conj} is fairly straightforward when $k$ or $|A|$ is even, as we now show. (But the odd case seems likely to be harder.)

\begin{theorem}\label{thm:Eulerian 2k-flow-connetced}
    For any positive integer $k$, every Eulerian graph is $2k$-flow-connected.
\end{theorem}

\begin{proof}
    Let $G$ be an Eulerian graph and let $(D,f)$ be a positive $2k$-flow on $G$. 
    Note that the set of edges of $G$ with even flow values induces an Eulerian subgraph $H$. By decomposing $H$ into edge-disjoint cycles, and subtracting $1$ along each of these cycles, we can change the flow values on all edges of $H$ to be odd. Hence, we transform $f$ into a flow $f'$ in which all flow values are odd and positive.

    Next, for each edge $e$ with $f'(e)=t$, we replace $e$ by a bundle $[e]$ of $t$ parallel edges, each with flow value $1$. Denote the resulting multigraph by $\hat{G}$. Now the flow $(D,f')$ naturally corresponds to an orientation $\hat{D}$ of $\hat{G}$ in which every edge has value $1$, and hence $\hat{D}$ is an Eulerian orientation of $\hat{G}$.
    By Lemma~\ref{flip-equiv-lem}, the orientation $\hat{D}$ can be transformed, via successive reversals of directed cycles, into an orientation $\hat{D}_1$ with the following property: for each bundle $[e]$ corresponding to an edge $e$ of $G$, all but one of the edges are paired in opposite directions (i.e., oriented one in and one out). This is feasible since each bundle $[e]$ contains an odd number of edges.

    The orientation $\hat{D}_1$ then naturally induces a flow $(D_1,f_1)$ on $G$, where $D_1$ is an Eulerian orientation and $f_1(e)=1$ for every edge $e\in E(G)$. Moreover, each cycle reversal in $\hat{G}$ corresponds to subtracting $2$ along a cycle in $G$, which is feasible since throughout these operations each flow value $f'(e)$ remains odd, so non-zero, and $|f'(e)|\le 2k-1$.
    Thus, every flow $(D,f)$ can be transformed into the flow $(D_1,f_1)$, where $f_1(e)=1$ for all edges $e$ and $D_1$ is Eulerian. By Lemma~\ref{flip-equiv-lem}, all Eulerian orientations are flip-equivalent, so $G$ is $2k$-flow-connected.
\end{proof}

Theorem~\ref{thm:Eulerian 2k-flow-connetced}, together with 
the sentence after Conjecture~\ref{conjA},
immediately implies that every Eulerian graph is $\Z_{2k}$-flow-connected. 
Moreover, we have that every Eulerian graph is $A$-flow-connected for each abelian group $A$ of even order.

\begin{theorem} For each abelian group $A$ and each positive integer $k$, every Eulerian graph $G$ is $(\mathbb{Z}_{2k}\times A)$-flow-connected. \end{theorem}

\begin{proof}
Fix a $(\mathbb{Z}_{2k}\times A)$-flow $(D,(f_1,f_2))$ on $G$, where $f_1$ is a $\mathbb{Z}_{2k}$-flow and $f_2$ is an $A$-flow.

We first modify $f_1$. Let $H$ be the subgraph induced by the edges $e$ for which $f_1(e)$ is even. Since $G$ is Eulerian, $H$ is also Eulerian and thus decomposes into edge-disjoint cycles. Assign an arbitrary Eulerian orientation to $H$, and add $1$ along each directed cycle. This operation yields a new $\mathbb{Z}_{2k}$-flow $f'$ satisfying $f'(e)\neq 0$ for every edge $e$. As in the previous proof, $f'$ is equivalent to a constant flow $f_1'\equiv 1$ under some Eulerian orientation $D'$ of $G$.

We next modify $f_2$. Let $H_2$ be the subgraph consisting of edges $e$ with $f_2(e)\neq 0$. Every edge of $H_2$ lies on a cycle of $H_2$. While $H_2$ is nonempty, choose a cycle $C$ in $H_2$, and let $a$ be the value of $f_2$ on some edge of $C$. Orient $C$ cyclically and subtract $a$ along $C$. This operation preserves the flow condition and reduces the number of nonzero edges in $f_2$. Repeating this process, we obtain the zero flow $f_2'$ of $G$.

Consequently, the original flow $(D,(f_1,f_2))$ is equivalent to $(D',(f_1',f_2'))$, where $f_1'\equiv 1$, $f_2'\equiv 0$, and $D'$ is any Eulerian orientation of $G$. Hence, 
by Lemma~\ref{flip-equiv-lem},
all $(\mathbb{Z}_{2k}\times A)$-flows are equivalent, and $G$ is $(\mathbb{Z}_{2k}\times A)$-flow-connected.
\end{proof}

\begin{obs}
\label{eulerian-obs}
	If $G$ is an Eulerian graph, then every edge-cut in $G$ has even order.
\end{obs}
\begin{proof}
	Since $G$ is Eulerian, $G$ decomposes into an edge-disjoint set $\mathcal{C}$ of cycles.  Consider an
	edge-cut $[S,\overline{S}]$ for some $S\subseteq V(G)$.  Each cycle $C\in \mathcal{C}$ must enter and
	leave $S$ an even number of times.  Since the cycles of $\mathcal{C}$ are edge-disjoint (with union
	equal to $E(G)$), we get that $|[S,\overline{S}]|$ is even, as claimed.
\end{proof}

As we mentioned above, the case of Conjecture~\ref{eulerian-conj} when $|A|$ is odd seems harder.  But we do know that it is true when $|A|$
is sufficiently large.

\begin{cor}
	If $G$ is an Eulerian graph and $A$ is an abelian group, then $G$ is $A$-flow-connected whenever
	$|A|\ge 62(56^2\cdot 27+27)+1=5,251,339$.
\end{cor}
\begin{proof}
	Our proof is by induction on $|V(G)|$, and our base case is when $G$ is 4-edge-connected, which is
	handled by Theorem~\ref{4-edge-connected-thm}.  So assume that $G$ is not 4-edge-connected.
	If $G$ is disconnected, then we finish by the induction hypothesis on each component.  So we instead
	assume that $G$ is connected.  By Observation~\ref{eulerian-obs}, every edge-cut of $G$ must be
	even; thus, $G$ has a 2-edge-cut $\{e_1,e_2\}$.  Form $G'$ from $G$ by contracting $e_2$.
	By the induction hypothesis, the result holds for $G'$.  Now we lift any reconfiguration sequence
	in $G'$ to a corresponding sequence in $G$: we ``uncontract'' $e_2$, orient $e_2$ oppositely from
	$e_1$, and at each step give $e_2$ the same flow value as $e_1$.  It is easy to check that this yields
	a valid flow reconfiguration sequence in $G$, which finishes the proof.
\end{proof}

\begin{obs}
For each abelian group $A$ with $|A|\ge 5$, every Eulerian graph is $A$-flow-connected if and only if every $4$-regular graph is $A$-flow-connected.
\end{obs}

\begin{proof}
It suffices to consider $4$-edge-connected graphs, since any $2$-edge-cut $\{e_1,e_2\}$ can be reduced by contracting one of the two edges, without affecting $A$-flow-connectedness.

We now reduce an Eulerian graph to a $4$-regular graph. Let $v$ be a vertex with $d(v)\ge 6$. 
Since $C_3$ is $A$-flow-reconfiguration-contractible (or $A$-FRC for short) whenever $|A|\ge 5$ by Corollary~\ref{AFRC-cor}\ref{C_3}, we can freely replace $v$ by a triangle without changing $A$-flow-connectedness. We replace $v$ by a triangle $v_1v_2v_3$ as follows. Partition the edges incident with $v$ into three nonempty sets $E_1,E_2,E_3$, where $|E_1|=|E_2|=2$ and $|E_3|=d(v)-4$. Assign the edges in $E_i$ to the vertex $v_i$ in the triangle for each $i\in\{1,2,3\}$, and make incident to $v_i$ each edge in $E_i$. Thus, $d(v_1)=d(v_2)=4$, while $d(v_3)=d(v)-2$.

Repeating this operation at each vertex of degree at least $6$, we eventually obtain a $4$-regular graph $G'$. Each replacement uses a $C_3$; since $C_3$ is $A$-FRC, each step preserves $A$-flow-connectedness. Thus, $G$ is $A$-flow-connected if and only if $G'$ is $A$-flow-connected.
\end{proof}

\begin{theorem}
Let $G$ be a $4$-regular graph, and let $A$ be an abelian group with $|A|\ge 5$. There exists a cubic bipartite graph $G'$ constructed from $G$ by a local expansion such that if $G'$ is $A$-flow-connected, then $G$ is also $A$-flow-connected.
\end{theorem}

\begin{proof}
First, observe that $G$ admits an orientation $D$ such that each vertex $v$ satisfies $d^+_D(v)=d^-_D(v)=2$.

We construct a cubic bipartite graph $G'$ from $G$ by replacing each vertex $v$ with a $4$-cycle $C_v$. Each cycle $C_v$ is given a fixed bipartition via a proper $2$-coloring, in which the vertices alternate between black and white along the cycle.

We distribute the four edges originally incident with $v$ according to the chosen orientation $D$: the two outgoing edges at $v$ are attached to distinct black vertices of $C_v$, and the two incoming edges are attached to distinct white vertices of $C_v$, so that each vertex of $C_v$ is incident with exactly one original edge. This is feasible since $d^+_D(v)=d^-_D(v)=2$.

\begin{figure}[!h]
\centering
\begin{tikzpicture}[thick, scale=1]
\tikzstyle{uStyle}=[shape = circle, minimum size = 5pt, inner sep = 0pt,
outer sep = 0pt, draw, fill=white, semithick]
\tikzstyle{bStyle}=[shape = circle, minimum size = 5pt, inner sep = 0pt,
outer sep = 0pt, draw, fill=black, semithick]
\tikzstyle{eStyle}=[shape = circle, minimum size = 5pt, inner sep = 0pt,
outer sep = 0pt, draw=none, fill=none, semithick]
\tikzset{every node/.style=uStyle}

\draw (0,0)   node[uStyle] (v) {};
\draw (v)     ++(0,-.25) node[eStyle] {\footnotesize{$v$}};
\draw (1,1)   node[eStyle] (E) {};
\draw (1,-1)  node[eStyle] (F) {};
\draw (-1,-1) node[eStyle] (G) {};
\draw (-1,1)  node[eStyle] (H) {};

\foreach \from/\to in {v/E, F/v, v/G, H/v}
	\draw[shorten <=2pt, shorten >=2pt, lmid arrow] (\from) -- (\to);

\begin{scope}[xshift = 4cm]
\draw (.5,.5)   node[bStyle] (A) {};
\draw (.5,-.5)  node[uStyle] (B) {};
\draw (-.5,-.5) node[bStyle] (C) {};
\draw (-.5,.5)  node[uStyle] (D) {};
\draw (1,1)   node[eStyle] (E) {};
\draw (1,-1)  node[eStyle] (F) {};
\draw (-1,-1) node[eStyle] (G) {};
\draw (-1,1)  node[eStyle] (H) {};
\draw (0,0)   node[eStyle]  {\footnotesize{$C_v$}};

\foreach \from/\to in {A/E, A/B, A/D, C/B, F/B, C/G, C/D, H/D}
	\draw[shorten <=2pt, shorten >=2pt, mid arrow] (\from) -- (\to);
\end{scope}

\end{tikzpicture}
	\captionsetup{width=.45\textwidth}
	\caption{A 4-vertex $v$ on the left, and the cycle $C_v$ that replaces $v$ on the right.}
\end{figure}
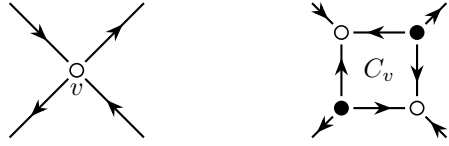

Next, we orient the edges of each $4$-cycle from black vertices to white vertices. Under this construction, every edge in the resulting graph $G'$ is directed from a black vertex to a white vertex, so $G'$ is cubic and bipartite.

Finally, since each expanded $4$-cycle is $A$-connected (that is, $C_4$ admits a $\beta$-flow for any $A$-boundary $\beta$) whenever $|A|\ge 5$, it follows from Theorem~\ref{AFRC-thm}\ref{AFRC-thm-1} that $A$-flow-connectedness of $G'$ implies $A$-flow-connectedness of $G$.
\end{proof}

Thus, for each odd $m\ge 5$, to prove that every Eulerian graph is $\mathbb{Z}_m$-flow-connected, it suffices to prove that every bridgeless cubic bipartite graph is $\mathbb{Z}_m$-flow-connected.

\section{\texorpdfstring{$\bm{A}$}{A}-Flow-Reconfiguration-Contractible Subgraphs}
\label{AFRC-sec}
In this section, we define a notion of contractible subgraphs $H$ for the property of a graph $G$ being $A$-flow-connected.
That is, $\F(G,A)$ is connected if and only if $\F(G/H,A)$ is.
Such a subgraph $H$ is called
\emph{$A$-flow-reconfiguration-contractible}, or \emph{$A$-FRC} for short; see Definition~\ref{$A$-FRC-def}.
Proving this reducibility is somewhat technical.  And because we handle all finite abelian groups $A$, not just $\Z_k$, we need to delve a bit more into group theory.

Our key lemma (Theorem~\ref{thm:cycle-FRC-exponent-threshold}) has three important consequences. First, we show that for each finite abelian group $A$, all $2$-edge-connected graphs are $A$-flow-connected if and only if this is true for all $2$-edge-connected \emph{cubic} graphs. (This mirrors many proofs of the $6$-Flow Theorem, which begin by reducing the problem to cubic graphs.) Second, we prove that each graph $G$ is $A$-flow-connected whenever $|A|=\Omega(\lg |G|)$. 
Third, we completely characterize the $A$-flow-connectedness of each graph containing a spanning triangularly connected subgraph, for each group $A$. To show the value of this lemma, we first prove these consequences in Section~\ref{sec4.1}, and later prove the core results in Section~\ref{sec4.2}.

\subsection{Consequences of \texorpdfstring{$\bm{A}$}{A}-Flow-Reduciblity}
\label{sec4.1}

A standard first step when proving a nowhere-zero flow theorem is showing that it suffices to prove it for $3$-edge-connected cubic graphs. \textbf{By using the results of the next section 
(Section~\ref{sec4.2}), we prove an analogous result for $\bm{A}$-flow-connectedness.}

\begin{theorem}\label{reduce-to-cubic-lem}
Fix an abelian group $A$ with $|A|\ge 5$. If there exists a $2$-edge-connected (multi)graph that is not $A$-flow-connected, then there exists a simple, 2-connected, $3$-edge-connected cubic graph $G$ that is not $A$-flow-connected.
\end{theorem}

\begin{proof}
    Assume the theorem is false. Let $\mathcal{G}$ 
    be the family of all 
    (multi)graphs that are
    2-edge-connected 
    but not $A$-flow-connected. We first choose a graph $G \in \mathcal{G}$ to minimize the number of edges $|E(G)|$. 
    Note that every graph is $A$-flow-connected if and only if all of its blocks are (since each cycle in $G$ is contained in exactly one of its blocks). Thus, we assume that $G$ is 2-connected. 
    
    First, suppose that $G$ has an edge $e$ of multiplicity at least 2. Form $G'$ from $G$ by contracting $e$. By Corollary~\ref{AFRC-cor}\ref{2K_2}, the multigraph $2K_2$ is $A$-FRC. By Theorem~\ref{AFRC-thm}(1), since $G$ is not $A$-flow-connected, neither is $G'$. Since $|E(G')|<|E(G)|$, this contradicts our minimal choice of $G$. Hence, $G$ must be a simple graph.

    Next, suppose that $G$ contains an edge-cut $\{e_1,e_2\}$. Fix an orientation of $G$ with $e_1$ and $e_2$ oriented toward opposite components of $G-\{e_1,e_2\}$. By flow conservation, every $A$-flow $f$ on $G$ must satisfy $f(e_1)=f(e_2)$. Form $G'$ from $G$ by contracting $e_2$. It is easy to check that $\mathcal{F}(G',A)\cong \mathcal{F}(G,A)$. Since $G$ is not $A$-flow-connected, neither is $G'$, which again contradicts the minimality of $G$. Thus, $G$ is 3-edge-connected.  

    Now, among all simple, 2-connected, and 3-edge-connected graphs in $\mathcal{G}$, we choose $G$ to minimize the excess degree sum $\Phi(G) = \sum_{v\in V(G)}\max\{d(v)-3, 0\}$. 
    (Our choice of $G$ here is independent of our choice of $G$ above.  But the arguments above prove that the set of graphs that we are minimizing over is non-empty.)
    We claim that $G$ must be cubic, meaning $\Phi(G) = 0$. Suppose, for a contradiction, that $G$ has a vertex $v$ with $d(v) \ge 4$. Let $d:=d(v)$.  We construct a new graph $G'$ from $G$ by deleting $v$ and introducing a 4-cycle $C_4$ on new vertices $v_1, v_2, v_3, v_4$. We distribute the $d$ original edges incident to $v$ by assigning exactly one edge to each of $v_1, v_2, v_3$, and the remaining $d-3$ edges to $v_4$. 

In this new graph $G'$, the degrees of $v_1, v_2, v_3, v_4$ are $3, 3, 3$, and $d-1$, respectively. Since $|A| \ge 5$, the cycle $C_4$ is $A$-connected. By Theorem~\ref{AFRC-thm}\ref{AFRC-thm-1}, if $G'$ were $A$-flow-connected, then $G$ would be as well, which is false. Thus $G' \in \mathcal{G}$. It is straightforward to see that this vertex expansion preserves simplicity, 2-connectedness, and 3-edge-connectedness. However, the contribution to the excess degree sum from $v$ (which was $d-3$) is now replaced by the contribution from $v_4$ (which is $d-4$), yielding $\Phi(G') = \Phi(G) - 1$. This contradicts the minimality of $G$. Hence, $G$ must be cubic.
\end{proof}

In fact, something slightly stronger is true.
\begin{obs}
The hypothesis $|A|\ge 5$ in Theorem~\ref{reduce-to-cubic-lem} can be removed.
\end{obs}
\begin{proof} 
To justify this observation, we simply consider each abelian group $A$ with $|A|\le 4$. When $A = \mathbb{Z}_2$, every graph $G$ has at most one nowhere-zero $A$-flow, so $\mathcal{F}(G,A)$ is trivial. When $A = \mathbb{Z}_3$, Theorem~\ref{cor:3-flow} shows that $\mathcal{F}(G,\mathbb{Z}_3)$ is connected only if $G$ is Eulerian. Thus, for any cubic graph $G$, $\mathcal{F}(G,\mathbb{Z}_3)$ is disconnected. For $A = \mathbb{Z}_4$, Esperet et al.~\cite[Section~2.4]{esperet-et-al} noted that $\mathcal{F}(K_4,\mathbb{Z}_4)$ is disconnected. In fact, $K_4$ is the first in an infinite class of 3-connected cubic graphs that are not $\mathbb{Z}_4$-flow-connected.

Similar results hold for the Klein four-group, $\Z_2\times\Z_2$. 
For each 3-edge-colorable cubic graph $G$, Esperet et al.~\cite[Theorem~4.3]{esperet-et-al} proved that $\mathcal{F}(G,\mathbb{Z}_2\times\mathbb{Z}_2)$ is connected if and only if all proper 3-edge-colorings of $G$ are Kempe-equivalent. Since belcastro and Haas~\cite[Theorem~4.10]{bH} constructed 3-connected cubic graphs with arbitrarily many Kempe equivalence classes, there are 3-connected cubic graphs $G$ where $\mathcal{F}(G,\mathbb{Z}_2\times\mathbb{Z}_2)$ has many components. (Finally, for $A=\Z_5$, Esperet et al.~\cite{esperet-et-al}, with the help of Wigderson and Moraski, gave a 3-connected cubic graph $G$ where $\mathcal{F}(G,\mathbb{Z}_5)$ is disconnected.)
\end{proof}

To further improve the general bounds for bridgeless graphs, we introduce the notion of \emph{cycle degeneracy}. This parameter captures the structural property that a graph can be 
reduced to a single cycle by successively contracting short cycles.

\begin{definition}
For an integer $k \ge 2$, a graph $G$ is \emph{$k$-cycle-degenerate} if it can be reduced to a single cycle by successively contracting subgraphs that are cycles of length at most $k$. The \emph{cycle degeneracy} of $G$, denoted by $\mathrm{cd}(G)$, is the smallest integer $k$ for which $G$ is $k$-cycle-degenerate. Conventionally, if $G$ is already a cycle itself, then we regard it as $1$-cycle-degenerate and set $\mathrm{cd}(G)=1$.
\end{definition}

The power of this parameter lies in its ability to guarantee global flow-connectedness purely through local conditions. If a graph has $\mathrm{cd}(G) \le k$, then for any abelian group $A$ with $|A| \ge \lfloor \frac{3k}{2} \rfloor + 1$, the graph $G$ is unconditionally $A$-flow-connected. The reasoning is straightforward: the size condition on $|A|$ ensures that every contracted cycle of length at most $k$ is an $A$-FRC subgraph. Since contracting $A$-FRC subgraphs preserves $A$-flow-connectedness, the problem reduces to the final remaining cycle, which is trivially $A$-flow-connected. Thus, cycle degeneracy yields a clean, sufficient condition for $A$-flow-connectedness. Note that by Theorem \ref{thm:cycle-FRC-exponent-threshold}, if the cycle $C_n$ is $A$-flow-connected, then every cycle $C_k$ with $k \le n$ is also $A$-flow-connected.

\begin{corollary}\label{cor: cd-A-flow}
Let $G$ be a graph with $\mathrm{cd}(G) \le k$. If the cycle $C_k$ is $A$-FRC, then $G$ is $A$-flow-connected. In particular, this is true when $|A| \ge \lfloor \frac{3k}{2} \rfloor + 1$.
\end{corollary}

Esperet et al.~\cite{esperet-et-al} demonstrated that Klee graphs (with $K_4$ being the simplest example) are not $\mathbb{Z}_4$-flow-connected, yet they are $(\mathbb{Z}_2 \times \mathbb{Z}_2)$-flow-connected. Notably, every Klee graph is a cubic graph with $\mathrm{cd}(G) = 3$. In Section 4, we proved a stronger structural limitation: no cubic graph containing a triangle can be $\mathbb{Z}_4$-flow-connected. On the other hand, our corollary guarantees that every graph with $\mathrm{cd}(G) \le 3$ is $A$-flow-connected for each abelian group $A$ with $|A| \ge 5$ or $A \cong \mathbb{Z}_2 \times \mathbb{Z}_2$. 
Building upon our $A$-FRC framework, this yields an alternative proof of their results and extends them to a broader class of graphs.

To upper bound the cycle degeneracy of general graphs, we use a classical breadth-first search (BFS) argument showing that every graph with minimum degree at least three contains a cycle of length $O(\log n)$.

\begin{lemma}\label{lem:log-cycle}
Let $G$ be an $n$-vertex graph with minimum degree $\delta(G) \ge 3$. Then $G$ contains a cycle of length at most $\lfloor2\log_2(n+2)\rfloor-1$.
\end{lemma}

\begin{proof}
Let $g$ be the girth of $G$, and let $r:= \lfloor \frac{g-1}{2} \rfloor$. Fix an arbitrary vertex $v \in V(G)$ and perform a breadth-first search (BFS) starting from $v$. 
Let $L_i$ denote the set of vertices at distance exactly $i$ from $v$. We claim that for all $0 \le i < r$, the forward branching is perfect; that is, no two vertices in $L_i$ share a common neighbor in $L_{i+1}$. Indeed, if two distinct paths from $v$ merged at some vertex $w \in L_{i+1}$, then $G$ would contain a cycle of length at most $2(i+1) \le 2r \le g-1 < g$, which contradicts that $G$ has girth $g$. 

Consequently, the BFS tree expands perfectly without any vertex overlap up to depth $r$. Since $\delta(G) \ge 3$, the root $v$ has $|L_1| \ge 3$. For each subsequent level $i\in[r-1]$, each vertex in $L_i$ has at least 2 distinct neighbors in $L_{i+1}$ (since its degree is at least 3, and exactly 1 edge connects to its parent in $L_{i-1}$). 
Thus, the number of vertices at level $i$ satisfies $|L_i| \ge 3 \cdot 2^{i-1}$ for all $i\in[r]$. Hence, the total number of vertices in $G$ must be at least the number of vertices in these first $r$ levels:
\[ n \ge 1 + \sum_{i=1}^{r} |L_i| \ge 1 + 3 \sum_{i=0}^{r-1} 2^i = 3 \cdot 2^r - 2. \]

This inequality simplifies to $2^r \le \frac{n+2}{3}$, and taking the base-2 logarithm yields:
\[ r \le \log_2(n+2) - \log_2 3. \]
We have
$r = \lfloor \frac{g-1}{2} \rfloor \ge \frac{g-2}{2}$, 
so
$g \le 2r + 2$. 
Thus,
\[ g \le 2\log_2(n+2) - 2\log_2 3 + 2. \]
Since $2\log_2 3 \approx 3.17 > 3$, 
we have $g < 2\log_2(n+2) - 1$. Since $g$ must be an integer, $G$ contains a cycle of length at most $\lfloor2\log_2(n+2)\rfloor-1$.
\end{proof}

Let $n_2$ denote the number of degree $2$ vertices in a bridgeless graph $G$. We can  suppress these vertices by replacing each degree $2$ vertex and its two incident edges with a single edge. This smoothing operation does not alter the $A$-flow-connectedness of the graph; that is, the connectedness of the reconfiguration graph $\mathcal{F}(G,A)$ remains unaffected. 
By applying this operation, the smoothed graph has $n-n_2$ vertices and minimum degree at least $3$. Combining this reduction with 
Corollary~\ref{cor: cd-A-flow} and
Lemma~\ref{lem:log-cycle},
we obtain the following logarithmic bound for general bridgeless graphs.

\begin{theorem}\label{thm:bridgeless-log-bound}
    Let $G$ be a bridgeless graph on $n$ vertices, exactly $n_2$ of which have degree $2$. Now $G$ is $A$-flow-connected for every abelian group $A$ satisfying
    \[ |A| \ge \left\lfloor 3\log_2(n-n_2+2) \right\rfloor. \]
\end{theorem}

For comparison, Esperet et al.~\cite{esperet-et-al} established an absolute constant bound, proving that every bridgeless graph is $A$-flow-connected whenever $|A| \ge (56^8\cdot 27+28)^{45}+1 \approx 1.15 \times 10^{694}$. Our logarithmic bound in Theorem~\ref{thm:bridgeless-log-bound}, while dependent on the graph's order, provides a smaller threshold unless the number of vertices is extremely large.

\subsection{\texorpdfstring{$\bm{A}$}{A}-Flow-Reducibility of Cycles}
\label{sec4.2}
We begin by defining some notation and terminology.
Let $A$ be an abelian group. A function $\beta:V(G)\to A$ is called an \emph{$A$-boundary} of $G$ if $\sum_{v\in V(G)} \beta(v)=0_A.$
For a graph $G$ and any $A$-boundary $\beta$ of $G$, a pair $(D, f)$ is called a \emph{nowhere-zero $\beta$-flow} (or simply \emph{$\beta$-flow}) if $D$ is an orientation on $G$ and $f:E(G)\to A\setminus\{0_A\}$ satisfies $$
\sum_{e\in E_D^+(v)} f(e)-\sum_{e\in E_D^-(v)} f(e)=\beta(v).$$
We define corresponding reconfiguration terms analogously to those introduced in
Section~\ref{intro-sec}: \emph{$\beta$-flow-adjacent}, \emph{$\beta$-flow-equivalent}, and \emph{$\beta$-flow-connected}.
A well-known property in flow theory states that if a graph $G$ contains a subgraph $H$ that admits a $\beta$-flow for every $A$-boundary $\beta$ of $H$ (in which case $H$ is called \emph{$A$-connected}), then $G$ admits an $A$-flow if and only if $G/H$ does. In the reconfiguration setting, however, merely preserving the existence of flows is insufficient. This motivates the following definitions.

\begin{definition}
Let $H$ be a graph, and let $\beta_1$ and $\beta_2$ be $A$-boundaries of $H$ that differ only at two vertices $v_1$ and $v_2$. We say that $H$ is \emph{$(\beta_1,\beta_2)$-flow-convertible} if there exist a $\beta_1$-flow $(D,f_1)$ and a $\beta_2$-flow $(D,f_2)$ such that $f_1-f_2$ is supported on a path whose endpoints are $v_1$ and $v_2$.
\end{definition}

\begin{definition}
\label{$A$-FRC-def}
Let $H$ be a graph and $A$ be an abelian group.
We say that $H$ is \emph{$A$-flow-reconfiguration-contractible}, or simply \emph{$A$-FRC}, if the following 3 conditions hold:
\begin{enumerate}[label=(\roman*)]
\item\label{FRC-1}
The graph $H$ is $A$-connected.
\item\label{FRC-2}
For every $A$-boundary $\beta$ of $H$, the graph $H$ is $\beta$-flow-connected.
\item\label{FRC-3}
For any two $A$-boundaries $\beta_1$ and $\beta_2$ that differ only at two vertices $v_1$ and $v_2$, the graph $H$ is $(\beta_1,\beta_2)$-flow-convertible.
\end{enumerate}
\end{definition}

We define $\F^*(G,A)$ to be the graph with the same vertex set as the flow reconfiguration graph $\F(G,A)$, but with a different adjacency relation: two vertices $(D,f_1)$ and $(D,f_2)$ are adjacent in $\F^*(G,A)$ if $f_1-f_2$ is supported on an even subgraph, rather than merely on a cycle, as in $\F(G,A)$. In fact, $\F^*(G,A)$ can be obtained from $\F(G,A)$ by taking every path $f_0 f_1 \dots f_k$ in which the difference between consecutive flows is supported on pairwise edge-disjoint cycles, and adding extra edges to turn such a path into a clique. It is easy to see that two flows $f_1$ and $f_2$ lie in the same component of $\F(G,A)$ if and only if they lie in the same component of $\F^*(G,A)$. Hence, $\F(G,A)$ is connected if and only if $\F^*(G,A)$ is connected.

Similarly, for a given $A$-boundary $\beta$ of $G$, we define $\F(G,\beta)$ and $\F^*(G,\beta)$ as the graphs whose vertices are all $\beta$-flows of $G$, where two vertices are adjacent if they differ on a cycle or on an even subgraph, respectively.

The following theorem shows that $A$-flow-connectedness is preserved under contraction of an $A$-FRC subgraph.
This result is what motivates
Definition~\ref{$A$-FRC-def}.

\begin{theorem}
\label{AFRC-thm}
Let $G$ be a graph, $H$ be a subgraph of $G$, and $A$ be an abelian group.
\begin{enumerate}[label=(\arabic*)]
    \item\label{AFRC-thm-1} If $H$ is $A$-connected, then $G/H$ is $A$-flow-connected whenever $G$ is $A$-flow-connected.
    \item\label{AFRC-thm-2} Moreover, if $H$ is $A$-FRC, then $G/H$ is $A$-flow-connected if and only if $G$ is $A$-flow-connected.
\end{enumerate}
\end{theorem}

The idea of the proof of Theorem~\ref{AFRC-thm} is very natural, but making this intuition precise requires some details.
We show for (1) that walks in $\F(G,A)$ map to walks in $\F(G/H,A)$ and for (2) that walks in
$\F(G/H,A)$ lift to walks in $\F(G,A)$.

\begin{proof}
Fix an orientation $D$ of $G$, and let $D':= D|_{E(G)\setminus E(H)}$.
That is, $D'$ is
the restriction of $D$ to $G/H$. 
(All flows we consider are with respect to these orientations.)

We first prove \ref{AFRC-thm-1}. We assume $H$ is induced: if $H$ is $A$-connected, then so is the induced subgraph $G[V(H)]$, and $G/H = G/G[V(H)]$ after deleting loops. 
Let $f$ be an arbitrary $A$-flow of $G$. Define a map $\pi: V(\F^*(G,A)) \to V(\F^*(G/H,A))$ by $\pi(f) = f'$, where $f'$ is obtained from $f$ by contracting $H$ to a single vertex $v_H$, while preserving the orientation and the flow value of each edge in $E(G) \setminus E(H)$. It is easy to check that $f'$ is an $A$-flow of $G/H$; so $\pi$ is well-defined.
It is well known that, since $H$ is $A$-connected, for each $A$-flow $f'$ of $G/H$, there exists an $A$-flow $f$ of $G$ with $\pi(f) = f'$. Thus $\pi$ is surjective.  

Now let $f_1, f_2$ be two adjacent vertices in $\F^*(G,A)$. Then the difference $f_1 - f_2$ is supported on an even subgraph $C$ of $G$. If $C \subseteq H$, then $\pi(f_1) = \pi(f_2)$. Otherwise, the image of $C$ in $G/H$ is also an even subgraph, denoted $C'$. In this case, $\pi(f_1) \neq \pi(f_2)$, and the difference $\pi(f_1) - \pi(f_2)$ is supported on $C'$. Hence $\pi(f_1)$ and $\pi(f_2)$ are adjacent in $\F^*(G/H, A)$. Consequently, $\pi$ induces a surjective homomorphism from $\F^*(G,A)$ to $\F^*(G/H, A)$. Therefore, if $G$ is $A$-flow-connected, then so is $G/H$. This proves \ref{AFRC-thm-1}.

We now prove \ref{AFRC-thm-2}. Sufficiency follows directly from \ref{AFRC-thm-1}. It remains to prove necessity. Let $\pi$ be the surjective homomorphism from $\F^*(G,A)$ to $\F^*(G/H,A)$ defined above.

Assume that $G/H$ is $A$-flow-connected. Fix an $A$-flow $f'$ on $G/H$, and choose two arbitrary flows $f_1, f_2 \in \pi^{-1}(f')$. Now $f_1(e) = f_2(e)$ for every edge $e \in E(G) \setminus E(H)$. We first show that $f_1$ and $f_2$ are $A$-flow-equivalent. Let $\beta$ be the $A$-boundary on $H$ induced by $f_1$ (equivalently by $f_2$). By Definition~\ref{$A$-FRC-def}\ref{FRC-2}, $\mathcal{F}(H,\beta)$ contains a path $g_1 g_2 \dots g_m$ such that $g_1 = f_1|_H$ and $g_m = f_2|_H$. For each $i \in [m]$, let $f^i$ be the flow on $G$ that agrees with $f_1$ on $E(G) \setminus E(H)$ and with $g_i$ on $E(H)$. Then $f^1 f^2 \dots f^m$ is a path in $\mathcal{F}(G,A)$ with $f^1 = f_1$ and $f^m = f_2$. Hence $f_1$ and $f_2$ are $A$-flow-equivalent.

Next, we show that if $f_1'$ and $f_2'$ are $A$-flow-adjacent in $G/H$, then there exist $f_1 \in \pi^{-1}(f_1')$ and $f_2 \in \pi^{-1}(f_2')$ such that $f_1$ and $f_2$ are $A$-flow-adjacent in $G$. Let $\operatorname{supp}(f_1' - f_2') = E(C')$ and consider two cases.

\textit{Case 1:} If $v_H \notin V(C')$, then $C'$ is a cycle in $G$ (as $v_H$ is the contracted vertex). Choose any $f_1 \in \pi^{-1}(f_1')$, and define $f_2$ by letting $f_2$ agree with $f_1$ on $E(H)$ and with $f_2'$ on $E(G) \setminus E(H)$. Then $\pi(f_2) = f_2'$, and $\operatorname{supp}(f_1 - f_2) = E(C')$. Therefore, $f_1$ and $f_2$ are $A$-flow-adjacent in $G$.

\begin{figure}[!h]
\tikzstyle{uStyle}=[shape = circle, minimum size = 5pt, inner sep = 0pt,
outer sep = 0pt, draw, fill=white, semithick]
\tikzstyle{bStyle}=[shape = circle, minimum size = 5pt, inner sep = 0pt,
outer sep = 0pt, draw, fill=black, semithick]
\tikzstyle{eStyle}=[shape = circle, minimum size = 5pt, inner sep = 0pt,
outer sep = 0pt, draw=none, fill=none, semithick]
\tikzset{every node/.style=uStyle}
\centering
\begin{tikzpicture}[semithick]
\draw (0,0) ellipse (1.5cm and 1cm);
\draw (.2,0) ellipse (.8cm and .35cm);
\draw[very thick] (.6,-.275) circle (.15cm);
\draw (-.85,0) node[eStyle] {$C$};
\draw (.6,-.65) node[eStyle] {$H$};
\draw (-1.35,-0.9) node[eStyle] {$G$};

\begin{scope}[xshift=-4cm]
\draw (0,0) ellipse (1.5cm and 1cm);
\draw (.45,0) ellipse (.6cm and .35cm);
\draw[line width=.6mm] (.6,-.335) circle (.035cm);
\draw (-.425,0) node[eStyle] {$C'$};
\draw (.72,-.6) node[eStyle] {$v_H$};
\draw (-1.62,-0.9) node[eStyle] {$G/H$};
\end{scope}
\end{tikzpicture}
	\captionsetup{width=.550\textwidth}
	\caption{
When lifting a reconfiguration sequence from $G/H$ to $G$ 
the most interesting case is when the cycle $C'$ in a reconfiguration step 
intersects $v_H$.}
\end{figure}
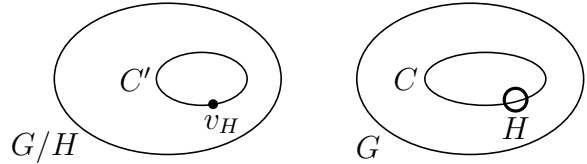

\textit{Case 2:} If $v_H \in V(C')$, then either $C'$ is a cycle in $G$, which can be handled as in the previous case, or $C'$ corresponds to a path $P$ in $G$ with endpoints $x, y \in V(H)$. Let $\beta_1$ and $\beta_2$ be the boundaries on $H$ induced by $f_1'$ and $f_2'$, respectively. Note that $\beta_1$ and $\beta_2$ differ only at vertices $x$ and $y$. By Definition~\ref{$A$-FRC-def}\ref{FRC-3}, there exist a $\beta_1$-flow $\psi_1$ and a $\beta_2$ flow $\psi_2$ such that $\operatorname{supp}(\psi_1 - \psi_2) = E(Q)$, where $Q$ is a path in $H$ with endpoints $x$ and $y$. For each $i \in [2]$, we obtain $f_i$ by extending $f_i'$ to $G$ so that it agrees with $\psi_i$ on $E(H)$. Then $\pi(f_i) = f_i'$, and $\operatorname{supp}(f_1 - f_2) = E(P) \cup E(Q)$. Since $P \cup Q$ forms a cycle, $f_1$ and $f_2$ are $A$-flow-adjacent in $G$.

Thus, in both cases, any two adjacent flows in $G/H$ lift to adjacent flows in $G$. Because $\pi$ is surjective, and any two flows in the same preimage $\pi^{-1}(f')$ are $A$-flow-equivalent, we conclude that $G$ is $A$-flow-connected. This proves \ref{AFRC-thm-2}.
\end{proof}

Motivated by Theorem~\ref{AFRC-thm}, we now turn our attention to cycles. 
Recall that for a nonzero element $\gamma$ in a finite abelian group $A$ (written additively), the \emph{order} of $\gamma$, denoted by $\ord(\gamma)$, is the smallest positive integer $k$ such that $k\gamma = 0$. 

\begin{remark}
It is helpful to visualize the action of adding $\gamma$ on the group $A$. 
Let $p:=\ord(\gamma)$.
The element $\gamma$ generates a cyclic subgroup $\langle \gamma \rangle$ of size $p$. This naturally partitions the group $A$ into $|A|/p$ disjoint cosets. Every coset has the same size $p$ and can be explicitly enumerated as $\{a, a+\gamma, a+2\gamma, \dots, a+(p-1)\gamma\}$ for some element $a \in A$.
\end{remark}

The following theorem provides a complete characterization of the $A$-FRC property for cycles over any finite abelian group $A$.

\begin{theorem}\label{thm:universal-FRC}
Let $A$ be a finite abelian group, 
and fix an integer $n \ge 2$. 
For each element $\gamma \in A \setminus \{0_A\}$, denote its order by $\ord(\gamma)$. The \emph{structural cost} $M(\gamma)$ is given by 
\[ M(\gamma):= \frac{|A|}{\ord(\gamma)} \left\lceil \frac{2 \cdot \ord(\gamma)}{3} \right\rceil. \]
The cycle $C_n$ is $A$-FRC if and only if $M(\gamma) \ge n + 1$ for every $\gamma \in A \setminus \{0_A\}$.
\end{theorem}

\begin{proof}
We begin by algebraically
parameterizing all flows on a cycle $C_n$. We identify all orientations of $C_n$ with a fixed cyclic orientation $D$ by replacing the value on an oppositely oriented edge by its inverse in $A$. For each $A$-boundary $\beta$, a $\beta$-flow $(D, f)$ on $C_n$ is uniquely determined by the flow value $a = f(e_1)$ on an arbitrary reference edge $e_1$. By flow conservation, the flow on each other edge $e \in E(C_n)$ can be written as $f(e) = a + c_e$, where $c_e \in A$ is a constant offset determined entirely by $\beta$.
The nowhere-zero condition requires that $a+c_e \neq 0_A$ for every edge $e\in E(C_n)$; equivalently, the initial value $a$ must avoid the forbidden values $-c_e$ for all $e\in E(C_n)$.

We first prove the sufficiency. Assume $M(\gamma) \ge n + 1$ for all $\gamma \neq 0_A$. 
Let $p:=\ord(\gamma)$.
Now $M(\gamma) = \frac{|A|}{p}\lceil \frac{2p}{3} \rceil \le \frac{|A|}{p} \cdot p = |A|$. Thus, our hypothesis gives $|A| \ge n+1$.

We verify the three conditions of Definition~\ref{$A$-FRC-def}. For each $A$-boundary $\beta$, the nowhere-zero condition forbids at most $n$ initial values (one for each edge). Since $|A| \ge n+1 > n$, there is at least one valid nowhere-zero $\beta$-flow; hence, $C_n$ is $A$-connected (so satisfies Definition~\ref{$A$-FRC-def}\ref{FRC-1}). Furthermore, every two $\beta$-flows differ by a constant flow along the entire cycle. Thus, their difference is a valid cycle flow, meaning the reconfiguration graph $\mathcal{F}(C_n, \beta)$ is connected, fulfilling Definition~\ref{$A$-FRC-def}\ref{FRC-2}.

To verify that $C_n$ is
$(\beta_1, \beta_2)$-convertible (Definition~\ref{$A$-FRC-def}\ref{FRC-3}), let $\beta_1$ and $\beta_2$ differ by $\gamma \in A \setminus \{0_A\}$ at vertices $u$ and $w$. Let $P_1$ and $P_2$ be the two paths on $C_n$ connecting $u$ and $w$. Note that $|E(P_1)| + |E(P_2)| = n$. Let $X$ be the set of forbidden initial flow values associated with $P_1$ under $\beta_1$ (i.e., $X = \{-c_e \mid e \in E(P_1)\}$), and $Y$ be the forbidden set for $P_2$. Clearly, $|X| \le |E(P_1)|$ and $|Y| \le |E(P_2)|$, yielding $|X| + |Y| \le n$. 

Let $Z:=A\setminus (X\cup Y)$.
To convert $\beta_1$ to $\beta_2$, we must either add $\gamma$ to $P_1$ or subtract $\gamma$ from $P_2$. Suppose  that instead both options fail to produce a nowhere-zero $\beta_2$-flow for every valid initial value $z \in Z$. If adding $\gamma$ to $P_1$ fails, then the shifted values must hit the forbidden set of $P_1$, implying $Z + \gamma \subseteq X$. If subtracting $\gamma$ from $P_2$ fails, then $Z - \gamma \subseteq Y$. 

Note that the element $\gamma$ generates a cyclic subgroup $\langle \gamma \rangle$ of order $p$,
which partitions $A$ into $|A|/p$ disjoint cosets. For any such coset $A_i$, let $Z_i:= Z \cap A_i$, $X_i:= X \cap A_i$, and $Y_i:= Y \cap A_i$. The failure conditions apply locally: $Z_i + \gamma \subseteq X_i$ and $Z_i - \gamma \subseteq Y_i$. Let $z_i:= |Z_i|$, $x_i:= |X_i|$, and $y_i:= |Y_i|$. These subset inclusions force $x_i \ge z_i$ and $y_i \ge z_i$, yielding $x_i + y_i \ge 2z_i$. Moreover, since $Z_i$ is the complement of $X_i \cup Y_i$ within the coset, we have $|X_i \cup Y_i| = p - z_i$, forcing $x_i + y_i \ge p - z_i$. Thus, for every coset, we have 
\[ x_i + y_i \ge \max(2z_i, p - z_i). \]
The lower bound for this maximum over all integers $z_i \ge 0$ occurs near the real-valued intersection $2z_i = p - z_i$, meaning $z_i \approx p/3$. Letting $z_i:= \lfloor p/3 \rfloor$ yields the exact minimum $\min(x_i + y_i) = \lceil 2p/3 \rceil$. Summing this minimum bound over all $|A|/p$ cosets gives:
\[ |X| + |Y| \ge \sum_{i=1}^{|A|/p} \left\lceil \frac{2p}{3} \right\rceil = \frac{|A|}{p} \left\lceil \frac{2p}{3} \right\rceil = M(\gamma). \]
However, this implies $M(\gamma) \le |X| + |Y| \le n$, which directly contradicts our initial hypothesis that $M(\gamma) \ge n+1$. Therefore, at least one conversion shift must succeed for some valid nowhere-zero flow, proving that $C_n$ is $A$-FRC.
\medskip

We now prove the necessity by an explicit construction. Suppose that there exists an element $\gamma \in A \setminus \{0_A\}$ such that $M(\gamma) \le n$. We will construct a specific $A$-boundary $\beta_1$, and assign a target boundary $\beta_2$ that differs from $\beta_1$ by exactly $\gamma$ at two distinct vertices. We will demonstrate that either no valid $\beta_1$-flow exists, or that no valid $\beta_1$-flow can be successfully reconfigured into any valid $\beta_2$-flow. In either case, this violates the conditions of Definition~\ref{$A$-FRC-def}, proving that $C_n$ is not $A$-FRC. 

Let $p := \ord(\gamma)$. In each coset $A_i$ of the cyclic subgroup $\langle \gamma \rangle$, we can select sets $X_i, Y_i \subseteq A_i$ satisfying $|X_i| + |Y_i| = \lceil 2p/3 \rceil$ such that for any element $z \in A_i \setminus (X_i \cup Y_i)$, we have $z + \gamma \in X_i$ and $z - \gamma \in Y_i$. 

To see this combinatorial fact explicitly, we identify $A_i$ with the cyclic sequence $(v_0, v_1, \dots, v_{p-1})$, where $v_{j+1} = v_j + \gamma$ (indices taken modulo $p$). We can systematically place valid initial values by defining the set $Z_i = \{v_{3m} \mid 0 \le m < \lfloor p/3 \rfloor\}$, which has size $\lfloor p/3 \rfloor$. The required conditions dictate that the right neighbors $v_{3m+1}$ must belong to $X_i$ and the left neighbors $v_{3m-1}$ must belong to $Y_i$. Since these neighbor sets are mutually disjoint for $m < \lfloor p/3 \rfloor$, we strictly assign them to $X_i$ and $Y_i$, respectively. Any remaining unassigned elements in $A_i$ can be arbitrarily partitioned into $X_i$ and $Y_i$. This construction leaves exactly $\lfloor p/3 \rfloor$ elements in $Z_i = A_i \setminus (X_i \cup Y_i)$, ensuring that $|X_i| + |Y_i| = p - \lfloor p/3 \rfloor = \lceil 2p/3 \rceil$.

Taking the union across all $|A|/p$ cosets, we obtain global forbidden sets $X = \bigcup X_i$ and $Y = \bigcup Y_i$, with a total combined size of $|X| + |Y| = M(\gamma)$.
Because $M(\gamma) \le n$, the cycle $C_n$ has enough edges to accommodate these constraints. We partition $C_n$ into two paths, $P_1$ and $P_2$, both with endpoints $u$ and $w$, such that $|E(P_1)| \ge |X|$ and $|E(P_2)| \ge |Y|$. We then directly assign the offset $c_e$ to each edge $e$: for edges in $P_1$, we injectively assign the value $-x$ for every $x \in X$, and assign arbitrary values from $-X$ to any remaining edges. We apply the identical procedure for $P_2$, assigning $-y$ for every $y \in Y$. 

With all $c_e$ fixed, we define the boundary $\beta_1(v) = c_{e_{out}} - c_{e_{in}}$ for every vertex $v$, where $e_{in}$ and $e_{out}$ are the incoming and outgoing edges incident to $v$ under the chosen orientation $D$. This definition ensures that $\sum_{v \in V(C_n)} \beta_1(v) = 0_A$, making $\beta_1$ a valid boundary. 
Under this specific boundary $\beta_1$, if we choose an initial flow value $a = f_1(e_1)$, the flow on any arbitrary edge $e \in E(C_n)$ is exactly $f_1(e) = a + c_e$. To ensure $f_1$ is a valid nowhere-zero flow, the initial value $a$ must avoid $-c_e$ for all edges. By our assignment, the set of all such forbidden values is exactly $X \cup Y$. Thus, the set of valid initial values is precisely $Z = A \setminus (X \cup Y)$.

At this point, we define the target boundary $\beta_2$. Let $\mathbf1_v: V(C_n) \to \{0, 1\}$ denote the standard indicator function taking the value $1$ at vertex $v$ and $0$ elsewhere. Now let $\beta_2:= \beta_1 + \gamma \cdot \mathbf1_{u} - \gamma \cdot \mathbf1_{w}$. 

Note that the number of valid initial values for $\beta_1$ is $|Z| = |A| - M(\gamma)$. If $M(\gamma) = |A|$, then $Z = \emptyset$. This implies that no valid $\beta_1$-flow exists at all, meaning $C_n$ violates the fundamental $A$-connectedness property (Definition~\ref{$A$-FRC-def}\ref{FRC-1}), completing the proof immediately. Therefore, we may assume $M(\gamma) < |A|$, which guarantees that $Z$ is non-empty. Consequently, there exists at least one valid $\beta_1$-flow $f_1$, generated by some initial value $z \in Z$.

To convert this valid $f_1$ into a $\beta_2$-flow, we must modify the flow along the path between $u$ and $w$. Suppose we attempt to achieve this by adding $\gamma$ to the edges of $P_1$. The new flow $f_2$ on any edge $e \in E(P_1)$ would algebraically become:
\[ f_2(e) = f_1(e) + \gamma = (z + c_e) + \gamma = (z + \gamma) + c_e. \]
Recall our structural setup: because $z \in Z$, we guaranteed that the shifted value $(z + \gamma)$ falls strictly into the set $X$. Furthermore, because we injectively assigned the value $-x$ to the offsets on $P_1$ for \textit{every} element $x \in X$, there must exist a specific edge $e^* \in E(P_1)$ that received the exact offset $c_{e^*} = -(z+\gamma)$. Evaluating the new flow $f_2$ on this specific edge yields exactly zero:
\[ f_2(e^*) = (z + \gamma) + c_{e^*} = (z + \gamma) - (z + \gamma) = 0_A. \]
This forces a zero value on $e^*$, meaning the conversion along $P_1$ universally fails for all valid initial values $z \in Z$.

Alternatively, modifying $P_2$ requires adding $-\gamma$ to the edges of $P_2$. The new flow on any edge $e \in E(P_2)$ becomes $(z - \gamma) + c_e$. Our construction inherently forces $(z - \gamma) \in Y$, and since the offsets on $P_2$ perfectly cover the additive inverses of $Y$, there must exist an edge $e^{**} \in E(P_2)$ where the flow identically hits zero. 

Since modifying the flow along either path forces an edge to take a value of $0_A$, the flow cannot be converted. Consequently, $C_n$ is not $A$-FRC.
\end{proof}

Although Theorem~\ref{thm:universal-FRC} gives a complete answer, it is not necessary to check $M(\gamma)$ for every element. We can simplify the $A$-FRC condition by using the structure of $A$. By the fundamental theorem of finite abelian groups, every finite abelian group $A$ can be written as a direct product of cyclic groups:
\[ A \cong \mathbb{Z}_{p_1^{n_1}} \times \mathbb{Z}_{p_2^{n_2}} \times \dots \times \mathbb{Z}_{p_k^{n_k}}, \]
where $p_1, p_2, \dots, p_k$ are prime numbers (not necessarily distinct) and $n_i \ge 1$ for each $i \in \{1, \dots, k\}$. 
The \emph{exponent} of $A$, denoted by $\exp(A)$, is the least common multiple of the orders of all elements in $A$. From the decomposition above, the exponent can be calculated as:
\[ \exp(A) = \operatorname{lcm}(p_1^{n_1}, p_2^{n_2}, \dots, p_k^{n_k}). \]
A basic property of finite abelian groups is that $A$ always contains an element of order $\exp(A)$. In fact, the element $(1, 1, \dots, 1)$ in the direct product has exactly this order. For more details about this classification and the group exponent, see~\cite[Section 5.2, Theorem 5]{dummit2004abstract}, or see~\cite{wiki:abelian_group_classification, enwiki:1350147197} for a quick overview.

To minimize the cost $M(\gamma)$, we need to consider the ceiling function. Any extra value added by the ceiling function disappears completely if $3$ divides $\exp(A)$. If not, this extra value becomes as small as possible when the element's order is as large as possible, which is exactly $\exp(A)$. This gives us the following simple global rule.

\begin{theorem}[Characterization of $A$-FRC Cycles]\label{thm:cycle-FRC-exponent-threshold}
Let $n \ge 3$ be an integer, let $A$ be a finite abelian group,
and let $m$ be its exponent. The cycle $C_n$ is $A$-FRC if and only if 
\[ \frac{|A|}{m} \left\lceil \frac{2m}{3} \right\rceil \ge n + 1. \]
\end{theorem}
By Theorem~\ref{thm:universal-FRC}, we want to choose $\gamma\in A\setminus\{0_A\}$ to minimize $M(\gamma)$.  If there exists $\gamma$ with $\ord(\gamma)$ divisible by $3$, then clearly this $\gamma$ achieves the minimum.  If not, then intuitively we want $\gamma$ with $\ord(\gamma)$ as large as possible, since this minimizes the ratio $\frac{1}{\ord(\gamma)}\lceil\frac{2\ord(\gamma)}3\rceil$.  (A priori, we might need to consider separately those $\gamma$ with $\ord(\gamma)\equiv 1\pmod 3$
and with $\ord(\gamma)\equiv 2\pmod 3$.)  The theorem statement unifies these possibilities.

\begin{proof}
By Theorem~\ref{thm:universal-FRC}, $C_n$ is $A$-FRC if and only if $\min_{\gamma \neq 0_A} M(\gamma) \ge n+1$. 
For each element order $p$, let $R(p):= \frac{1}{p}\lceil \frac{2p}{3} \rceil$; this is the
\emph{cost ratio function}.
Let $\operatorname{Ord}(A) := \{ \ord(\gamma) \mid \gamma \in A \setminus \{0_A\} \}$; so $\operatorname{Ord}(A)$ is the set of all possible orders in $A$.
Since $M(\gamma) = |A| \cdot \frac{1}{\ord(\gamma)}\left\lceil \frac{2\cdot \ord(\gamma)}{3} \right\rceil$, minimizing this cost (in the statement of the theorem) over all non-zero elements is equivalent to finding the global minimum of $R(p)$
over all element orders $p$. 

We first establish an algebraic expansion of $R(p)$. Let $p = 3k + r$, where $r = p \bmod 3$; here $r \in \{0, 1, 2\}$. 
The ratio expands as:
\[ R(p) = \frac{1}{3k+r} \left\lceil \frac{6k+2r}{3} \right\rceil = \frac{2k + \lceil 2r/3 \rceil}{3k+r}. \]
By evaluating the ceiling function for $r \in \{0, 1, 2\}$, one can easily verify that for all integers $p \ge 2$, we have the identity:
\[ R(p) = \frac{2}{3} + \frac{p \bmod 3}{3p}. \]

A fundamental property of finite abelian groups is that there exists an element whose order $m$ is exactly the exponent $\exp(A)$. 
Thus, $m \in \operatorname{Ord}(A)$. Furthermore, $p|m$ for every element order $p \in \operatorname{Ord}(A)$. So $m = p \cdot d$ for some integer $d \ge 1$. We claim that 
$R(m) \le R(p)$ for all $p \in \operatorname{Ord}(A)$. 
This is equivalent to showing:
\[ \frac{2}{3} + \frac{m \bmod 3}{3m} \le \frac{2}{3} + \frac{p \bmod 3}{3p},
\]
which simplifies to
$m \bmod 3 \le d \cdot (p \bmod 3)$.

We verify this latter inequality:
If $3 \mid p$, then $p \bmod 3 = 0$. Since $p$ divides $m$, we must also have $3 \mid m$, yielding $m \bmod 3 = 0$. The inequality simplifies to $0 \le 0$, which holds.
If $3 \nmid p$, then $p \bmod 3 \ge 1$. If $d = 1$, then $m = p$ and the equality holds trivially. If $d \ge 2$, then the right side is at least $2 \times 1 = 2$. Since $m \bmod 3 \in \{0, 1, 2\}$, its maximum possible value is 2, ensuring that $m \bmod 3 \le 2 \le d \cdot (p \bmod 3)$ unconditionally.

Thus, the cost ratio is strictly minimized by the maximum order $m$. The global minimum structural cost over the group is exactly $ \frac{|A|}{m} \left\lceil \frac{2m}{3} \right\rceil$. 
Thus, the cycle $C_n$ is $A$-FRC if and only if this minimum cost is at least $n+1$. 
\end{proof}

Before studying specific types of groups, we give a general sufficient condition based only on the group size.

\begin{corollary}\label{cor:universal-sufficient}
    Fix an integer $n \ge 2$, and let $A$ be a finite abelian group. If $|A| \ge \lfloor \frac{3n}{2} \rfloor + 1$, then the cycle $C_n$ is $A$-FRC. In particular, when $A$ is the cyclic group $\mathbb{Z}_k$, $C_n$ is $\mathbb{Z}_k$-FRC if and only if $k \ge \lfloor \frac{3n}{2} \rfloor + 1$.
\end{corollary}

\begin{proof}
    Let $m:= \exp(A)$; this is the exponent of $A$. Since $m$ must divide the group order, there exists a positive integer $d$ such that $|A| = d \cdot m$. By Theorem~\ref{thm:cycle-FRC-exponent-threshold}, the structural cost of $A$ is precisely:
    \[\frac{|A|}{m} \left\lceil \frac{2m}{3} \right\rceil = d \left\lceil \frac{2m}{3} \right\rceil. \]
    For each real number $x$ and each positive integer $d$, 
    we have $d \lceil x \rceil \ge \lceil dx \rceil$. Applying this inequality to our cost function yields:
    \[ d \left\lceil \frac{2m}{3} \right\rceil \ge \left\lceil d \cdot \frac{2m}{3} \right\rceil = \left\lceil \frac{2|A|}{3} \right\rceil. \]
    Notice that the exponent of the cyclic group $\Z_{|A|}$ is simply the group's order $|A|$.  So $\lceil \frac{2|A|}{3} \rceil$ is exactly the structural cost of the cyclic group $\mathbb{Z}_{|A|}$. Thus, the structural cost of each finite abelian group is bounded below by the cost of the cyclic group of the same order. 
    
    To ensure $C_n$ is universally $A$-FRC, we require that $\lceil 2|A|/3\rceil \ge n+1$.
    Now the result follows by simple algebra.
    \[ \left\lceil \frac{2|A|}{3} \right\rceil \ge n+1 
    \iff \frac{2|A|}{3} > n 
    \iff |A| > \frac{3n}{2} 
    \iff |A| \ge \left\lfloor \frac{3n}{2} \right\rfloor + 1. \]
    \aftermath
\end{proof}
\bigskip

Theorem \ref{thm:cycle-FRC-exponent-threshold} directly gives the exact characterizations, of values of $n$ for which $C_n$ is $A$-FRC, for several common finite abelian groups $A$.

\begin{corollary}\label{cor:specific-groups}
    Let $n \ge 2$ be an integer. For the following finite abelian groups $A$, the cycle $C_n$ is $A$-FRC if and only if:
    \begin{itemize}
        \item $A \cong \mathbb{Z}_k$: \quad $|A| \ge \lfloor \frac{3n}{2} \rfloor + 1$.
        \item $A \cong \mathbb{Z}_2^k$: \quad $|A| \ge n + 1$. 
        \item $|A| \equiv 0 \pmod 3$: \quad $|A| \ge \left\lceil \frac{3(n+1)}{2} \right\rceil$.
    \end{itemize}
\end{corollary}
\begin{proof}
   The first point follows immediately from  
   Corollary~\ref{cor:universal-sufficient}.  The second follows from 
   Theorem~\ref{thm:cycle-FRC-exponent-threshold}, since $\Z_2^k$ has exponent $2$.
   The third point also follows from 
   Theorem~\ref{thm:cycle-FRC-exponent-threshold}, by substituting $m=3s$ for some integer $s$, and simplifying.
\end{proof}

Applying Corollary~\ref{cor:specific-groups} to small cycles ($n\in\{2, 3\}$) yields the explicit $A$-FRC thresholds for $2K_2$ and $C_3$. This naturally extends to the global flow-connectedness of any graph that contains a spanning triangularly connected subgraph. 
Here, a graph is \emph{triangularly connected} if for any two of its edges, there is a sequence of triangles $T_1, T_2, \dots, T_k$ such that the first triangle contains one edge, the last triangle contains the other, and any two consecutive triangles share at least one edge. A \emph{spanning} triangularly connected subgraph is simply a triangularly connected subgraph that contains all vertices of the graph. Common examples of triangularly connected graphs include complete graphs $K_n$ (for $n \ge 3$), wheel graphs, and maximal planar graphs. 

Note that $\mathbb{Z}_2$-flow-connectedness is trivial. Furthermore, we characterize $\mathbb{Z}_3$-flow-connectedness in Theorem~\ref{Z3-flow-main-thm}, and we show in Theorem~\ref{notZ4flow-connectedII} that no graph containing a triangle can be $\mathbb{Z}_4$-flow-connected. Together with the corollary below, this completely answers the $A$-flow-connectedness problem for such graphs over all groups $A$.

\begin{corollary}\label{AFRC-cor}
Let $A$ be an abelian group with $|A| \ge 4$. The following hold.
\begin{enumerate}[label=(\roman*)]
\item \label{2K_2}
The graph $2K_2$ is $A$-FRC. 

\item \label{C_3}
The cycle $C_3$ is $A$-FRC if and only if $|A| \ge 5$ or $A \cong \mathbb{Z}_2 \times \mathbb{Z}_2$. 

\item \label{triangle-connected subgraph}
Any graph containing a spanning triangularly connected subgraph is $A$-flow-connected if and only if $|A| \ge 5$ or $A\cong \mathbb{Z}_2 \times \mathbb{Z}_2$.
\end{enumerate}
\end{corollary}

\section{\texorpdfstring{$\bm{A}$}{A}-Flow-connectedness of 4-Edge-connected Graphs}
\label{4edge-sec}

Tutte conjectured that every $2$-edge-connected graph has a nowhere-zero $A$-flow for every finite abelian group $A$ with $|A|\ge 5$.  This was proved by Jaeger~\cite{jaeger} when $|A|\ge 8$ and by Seymour~\cite{seymour} when $|A|\ge 6$.  Tutte also conjectured that every $4$-edge-connected graph has a nowhere-zero $A$-flow whenever $|A|\ge 3$; Jaeger proved this when $|A|\ge 4$.
For flow reconfiguration, Esperet et al.~\cite{esperet-et-al} proved the following large-group theorem.

\begin{thmA}
    \label{esperet-main-thm}
    If $G$ is a $2$-edge-connected graph and $A$ is a finite abelian group, then $G$ is $A$-flow-connected whenever $|A|\ge (56^8\cdot 27+28)^{45}+1\approx 1.15\times 10^{694}.$
\end{thmA}

In this section we prove a stronger result for $4$-edge-connected graphs.  The proof has two independent parts.  First, a direct product argument shows that $G$ is $(A_1\times A_2)$-flow-connected whenever each factor has high enough ``local flow-choosability''.  Second, a separate argument handles groups with a sufficiently large cyclic factor of order a prime power.  Combining the two cases gives the following theorem.

\begin{theorem}
    \label{4-edge-connected-thm}
    If $G$ is a $4$-edge-connected graph and $A$ is a finite abelian group, then $G$ is $A$-flow-connected whenever
    $
        |A|\ge 62(56^2\cdot 27+27)+1=5,251,339.
    $
\end{theorem}

\subsection{Flow decompositions and dense subgraphs}

We first recall some basic facts on flow decomposition that we will use to construct our reconfiguration sequences. For a graph $G$, let $c(G)$ denote its number of components.

\begin{lem}
\label{decompose-flow-lem}
Let $G$ be a graph, let $A$ be an abelian group, and let $f$ be an $A$-flow on $G$ (not necessarily nowhere-zero).
Now $f$ can be written as a sum of at most $|E(G)|-|V(G)|+c(G)$ flows supported on cycles, where $c(G)$ is the number of components in $G$.
\end{lem}

\begin{proof}
Fix a spanning forest $F$ of $G$. As long as the support of the current flow contains a cycle $C$, choose an edge
$e\in E(C)\setminus E(F)$ and subtract from the current flow the cycle-flow on $C$ with value equal to the current flow value on $e$.
This gives $e$ flow value 0 and never changes the flow value on any edge of
$E(G)\setminus E(F)$ that already has flow value $0$.

Since $F$ is a spanning forest, it has $|V(G)|-c(G)$ edges. Hence
$|E(G)\setminus E(F)|=|E(G)|-|V(G)|+c(G)$.
Therefore after at most $|E(G)|-|V(G)|+c(G)$ steps, all edges outside $F$ have value zero.
The remaining support is contained in a forest, and hence is empty, since a forest admits no nonzero flow.
\end{proof}

\begin{lem}
    \label{decompose-flow-lem-II}
    Fix an integer $k\ge 2$.  Every integer $k$-flow $f$ in a graph $G$, with 0 values allowed, is the sum of integer $k$-flows $(f_i)_{1\le i\le s}$ supported on cycles such that, for every edge $e$ and every $j\in[s]$, the partial sum
    $
        \sum_{1\le i\le j} f_i(e)
    $
    lies between $0$ and $f(e)$.
\end{lem}
\begin{proof}
    We reverse the underlying orientation of every edge on which $f$ is negative, and negate the corresponding value.  Thus we assume that $f(e)\ge 0$ for all edges $e$.  If $f$ is not identically 0, then the support of $f$ containswa directed cycle $C$.  Subtract the unit flow on $C$ and repeat.  The values on each edge decrease monotonically from $f(e)$ to $0$. Now reversing all edges with initial flow value negative gives the desired decomposition for the original orientation.
\end{proof}

\begin{lem}
    \label{k-Zk-hom-lem}
    Let $f$ and $f'$ be two nowhere-zero integer $k$-flows, and let $g$ and $g'$ be their images as nowhere-zero $\Z_k$-flows.  If $f$ and $f'$ are in the same component of $\F(G,k)$, then $g$ and $g'$ are in the same component of $\F(G,\Z_k)$.
\end{lem}
\begin{proof}
    Taking every integer flow value modulo $k$ maps each valid reconfiguration step in $\F(G,k)$ either to a valid reconfiguration step in $\F(G,\Z_k)$ or to a loop.  Thus every walk in $\F(G,k)$ maps to a walk in $\F(G,\Z_k)$ after deleting loops.
\end{proof}

A \emph{$(\le r)$-subdivision} of a graph $H$ is formed from $H$ by replacing each edge by a path of length at most $r+1$. For a vertex subset $X$ of a graph $G$, let $\partial(X)$ denote the edge-cut between $X$ and $X^c$. The following lemma and its proof are inspired by Lemma~6.1 {in~\cite{esperet-et-al}}.

\begin{lem}[Dense subgraphs contain short subdivisions]
    \label{short-subdivision-density-lem}
    Fix an integer $r\ge 1$.  Let $G$ be a (multi)graph with $|V(G)|\ge r+1$.  If
    \begin{equation}
        |E(G)|>\frac{r+2}{r+1}|V(G)|-2,
        \label{density-ineq}
    \end{equation}
    then $G$ has either a $(\le 2r+1)$-cycle or a $(\le r)$-subdivision of a $3$-edge-connected graph.
\end{lem}

\begin{proof}
    Let $c:=(r+2)/(r+1)$.
	Throughout we assume that $G$ is a simple graph, since otherwise $G$ contains a 2-cycle.
	Our proof is by induction on $|V(G)|$.  If $|V(G)|=r+1$, then $|E(G)| > c(r+1)-2 = r = |V(G)|-1$; 
	so $G$ contains a $(\le r+1)$-cycle, and we are done. Thus, we assume that $|V(G)|\ge r+2$.  If $G$ contains a vertex $v$ with $d(v)\le 1$,
	then let $G':=G-v$.  Note that $|E(G')| \ge |E(G)|-1 > c(|V(G)|-1)-2 = c|V(G')|-2$, so we are done by
	induction.  Thus, we assume that $\delta(G)\ge 2$.  If $|V(G)|\le 2r+1$, then $G$ contains a $(\le 2r+1)$-cycle,
	and we are done; so we assume that $|V(G)|\ge 2r+2$.  So $|E(G)|> c|V(G)| -2 \ge |V(G)|$.  
	Hence, $G$ contains a cycle, but is more than just a cycle.

	If $G$ is disconnected,
	then by an averaging argument, some component $G'$ satisfies $|E(G')| > c|V(G')|-2$.  Since $\delta(G')\ge 2$, we get that $G'$ contains a cycle, so $|V(G')|\ge 2r+2$ and we are again
	done by the induction hypothesis.  Hence, we instead assume that $G$ is connected.
	Now suppose that $G$ contains an edge-cut $S$ with $|S|\le 2$.  Let $G_1$ be a component of $G-S$,
	and let $G_2:=(G-S)-G_1$.  If $|E(G_i)| \le c|V(G_i)|-2$, for each $i\in [2]$, then summing these
	inequalities gives $|E(G)| \le c|V(G)|-2$,  contradicting \eqref{density-ineq}.  So, we assume, by symmetry, that
	$|E(G_1)| > c|V(G_1)|-2$.  If $|V(G_1)|\ge r+1$, then we are done by applying the induction hypothesis 
	to $G_1$.  So we assume that $|V(G_1)|\le r$.  Since $\delta(G)\ge 2$, this implies that $|S|=2$.
	That is, $G$ is 2-edge-connected.  
    A \emph{thread} is a path whose internal vertices have degree 2 in $G$ and whose endpoints have degree at least 3. 
    Thus, the argument above shows that each thread has length at most $r+1$.

    It remains to show that $H$ is $3$-edge-connected.  Suppose not, and let $S$ be a nonempty proper subset of $V(H)$ with $|\partial_H(S)|\le 2$.  Since $\delta(H)\ge 3$, each of $H[S]$ and $H-S$ contains a cycle.  Lifting these two cycles to $G$ shows that both sides of the corresponding edge-cut of $G$ contain a cycle.  Since $G$ has no cycle of length at most $2r+1$, each side has more than $r$ vertices.  This contradicts the assumption that every edge-cut of size at most $2$ in $G$ has a side of order at most $r$.  Hence $H$ is $3$-edge-connected, and $G$ contains a $(\le r)$-subdivision of a $3$-edge-connected graph.
\end{proof}

\begin{cor}
    \label{sparse-good-subgraph-unified}
    Let $G$ be a multigraph with $|V(G)|\ge 2$. Now
    \begin{enumerate}[label=(\arabic*)]
        \item if $|V(G)|\ge 2$ and $|E(G)|>\frac{3}{2}|V(G)|-2$, then $G$ contains either a cycle of length at most $3$ or a $(\le 1)$-subdivision of a $3$-edge-connected graph;
        \item if $|V(G)|\ge 3$ and $|E(G)|>\frac{4}{3}|V(G)|-2$, then $G$ contains either a cycle of length at most $5$ or a $(\le 2)$-subdivision of a $3$-edge-connected graph.
    \end{enumerate}
\end{cor}

\begin{proof}
Apply Lemma~\ref{short-subdivision-density-lem} with $r=1$ in the first case and $r=2$ in the second case.
\end{proof}

\subsection{Local Flow-choosability}

For each abelian group $A$ and each integer $k\ge 1$, an oriented graph $G$ is \emph{$(A,k)$-flow-choosable} if, for every assignment $F:E(G)\to 2^A$ with $|F(e)|\le k$ for all $e$, there is an $A$-flow $f$ on $G$, with 0 values allowed, such that $f(e)\notin F(e)$ for every edge $e$.

We use two standard lemmas of Jaeger, Linial, Payan, and Tarsi~\cite{jaeger-et-al}.  We say that $G$ is the \emph{$2$-closure} of a spanning subgraph $R$ if there exist cycles $C_1,\ldots,C_t$ such that
\[
    E(G)=E(R)\cup E(C_1)\cup\cdots\cup E(C_t),
\]
where each $C_i$ contains at most two edges not already in $E(R)\cup E(C_1)\cup\cdots\cup E(C_{i-1})$.

\begin{lemA}[Jaeger--Linial--Payan--Tarsi~{\cite[Lemma~2.1]{jaeger-et-al}}]
    \label{two-closure-lem}
    Let $G$ be an oriented graph and let $T$ be a spanning tree of $G$ such that $G$ is the $2$-closure of the spanning subgraph with edge set $E(G)\setminus E(T)$.  Let $A$ be a finite abelian group.  If each edge $e\in E(T)$ has a forbidden set $F(e)\subseteq A$ of size less than $|A|/2$, then $G$ has an $A$-flow $f$ such that
    \[
        f(e)\notin F(e)\qquad\text{for every }e\in E(T).
    \]
\end{lemA}

\begin{lemA}[Jaeger--Linial--Payan--Tarsi~{\cite[Lemma~3.3]{jaeger-et-al}}]
    \label{truncated-cubic-tree-lem}
    If $G$ is obtained from a cubic $3$-edge-connected graph by deleting one vertex, then $G$ has a spanning tree $T$ such that $G$ is the $2$-closure of the spanning subgraph with edge-set $E(G)\setminus E(T)$.
\end{lemA}

\begin{lem}[Subgroup Flow-choosability]
    \label{subgroup-flow-choosable-lem}
    Fix an integer $k\ge 1$, and let $A$ be a finite abelian group.  Suppose that $A$ contains a subgroup $H$ such that
    \[
        |H|\ge k+1
        \qquad\text{and}\qquad
        {|A|/|H|}\ge 2k+1.
    \]
    Now for every $3$-edge-connected oriented graph $G$ and every vertex $v$ of degree $3$ in $G$, the graph $G-v$ is $(A,k)$-flow-choosable.  In particular, every $3$-edge-connected graph is $(A,k)$-flow-choosable.
\end{lem}
\begin{proof}
    The second statement follows from the first by adding a new vertex of degree $3$ to the given graph, allowing parallel edges if necessary, and then deleting it. To prove the first statement, we replace each vertex $w$ of degree at least $4$ by a cycle of length $d_G(w)$, connecting each edge incident to $w$ to a distinct vertex of the cycle. This reduction allows us to assume that $G$ is cubic. 

    Let $F(e)\subseteq A$ be forbidden sets with $|F(e)|\le k$.  By Lemma~\ref{truncated-cubic-tree-lem}, the graph $G-v$ has a spanning tree $T$ such that $G-v$ is the $2$-closure of the spanning subgraph with edge set $E(G-v)\setminus E(T)$.  For each edge $e\in E(T)$, let
    $
        F^*(e):=F(e)+H=\{z+h:z\in F(e),\ h\in H\}.
    $
    Since $|F(e)|\le k$ and $[A:H]\ge 2k+1$, we have
    $
        |F^*(e)|\le k|H|<\frac{|A|}{2}.$
    So Lemma~\ref{two-closure-lem} gives an $A$-flow $f$ such that
    $
        f(e)\notin F(e)+H
    $ {for every }$e\in E(T)$.
    
    Now fix $e\in E(G-v)\setminus E(T)$.  Since $|H|\ge k+1$, we may choose $x_e\in H$ such that
    \[
        f(e)+x_e\notin F(e).
    \]
    Add, for each non-tree edge $e$, the value $x_e$ on the fundamental cycle of $e$ with respect to $T$.  Let $g$ be the sum of these added cycle-flows.  Then $g(e)=x_e$ on each non-tree edge $e$, and $g(e')\in H$ on every tree edge $e'$.  Thus $f+g$ avoids $F(e)$ on every non-tree edge by the choice of $x_e$, and avoids $F(e')$ on every tree edge $e'$ because $f(e')\notin F(e')+H$ and $g(e')\in H$.  Hence $G-v$ is $(A,k)$-flow-choosable.
\end{proof}

We also use the following theorem of Langhede and Thomassen.

\begin{thmA}[Langhede--Thomassen~{\cite[Corollary 16]{LT}}]
    \label{LT-lem}
    Fix an integer $k\ge 1$, and let $A$ be an abelian group with $|A|>8k^3$.  If $G$ is $3$-edge-connected, then $G$ is $(A,k)$-flow-choosable.
\end{thmA}

\begin{lem}[Short Subdivisions]
    \label{local-subdivision-choosability-lem}
    Let $A$ be a finite abelian group, and let $r,\ell,k$ be positive integers.  Suppose that
    \begin{enumerate}[label=\textup{(\roman*)}]
        \item $|A|>k\ell$, and
        \item every $3$-edge-connected oriented graph is $(A,k(r+1))$-flow-choosable.
    \end{enumerate}
    Then every orientation of every cycle of length at most $\ell$, and every orientation of every $(\le r)$-subdivision of a $3$-edge-connected graph, is $(A,k)$-flow-choosable.
\end{lem}
\begin{proof}
    Let $H$ be the oriented graph under consideration, and let $F(e)\subseteq A$ be forbidden sets with $|F(e)|\le k$ for each edge $e$.

    First suppose that $H$ is a cycle of length at most $\ell$.  Fix a cyclic orientation of $H$.  A cycle-flow on $H$ is determined by a single flow value $x\in A$.  Each edge forbids at most $k$ choices of $x$, after possibly replacing each forbidden element by its inverse $0_A-x$ according to the orientation of the edge.  Hence at most $k\ell$ elements of $A$ are forbidden.  Since $|A|>k\ell$, some choice of $x$ avoids all forbidden values.

    Now suppose that $H$ is a $(\le r)$-subdivision of a $3$-edge-connected graph $J$.  Fix an arbitrary orientation of $J$.  For each edge $e\in E(J)$, let $P_e$ be the corresponding path of $H$.  A value assigned to $e$ induces, up to sign, the same value on every edge of $P_e$.  Thus the forbidden sets on the edges of $P_e$ induce a forbidden set $F_J(e)\subseteq A$ of size at most $k|E(P_e)|\le k(r+1)$, again after changing signs where orientations disagree.  By assumption, $J$ has an $A$-flow $x$ such that $x(e)\notin F_J(e)$ for every $e\in E(J)$.  Lifting $x$ to the subdivided paths of $H$ gives an $A$-flow on $H$ avoiding all sets $F(e)$.
\end{proof}

\begin{definition}
    \label{admissible-group-def}
    A finite abelian group $A$ is \emph{admissible} if the following 2 conditions hold.
    \begin{enumerate}[label=\textup{(\roman*)}]
        \item Every orientation of every $(\le 3)$-cycle is $(A,1)$-flow-choosable; and
        \item Every orientation of every $(\le 1)$-subdivision of a $3$-edge-connected graph, is $(A,1)$-flow-choosable.
    \end{enumerate}
\end{definition}
Note that (i) implies that every orientation of every 3-edge-connected graph is $(A,2)$-flow-choosable.
To see this, we simply subdivide each edge $e$ with $|F(e)|=2$ and let each of the resulting new edges inherit a distinct value in $F(e)$; afterwards, we apply (i).

\begin{cor}
    \label{admissible-cor}
    Let $A$ be a finite abelian group.  If either
    \begin{enumerate}[label=\textup{(\roman*)}]
        \item $|A|\ge 65$, or
        \item $A$ has a subgroup $H$ with $|H|\ge 3$ and $[A:H]\ge 5$,
    \end{enumerate}
    then $A$ is admissible.
\end{cor}
\begin{proof}
    If $|A|\ge 65$, then every $3$-edge-connected graph is $(A,2)$-flow-choosable by Theorem~\ref{LT-lem} with $k=2$.  If $A$ has a subgroup $H$ with $|H|\ge 3$ and $[A:H]\ge 5$, then every $3$-edge-connected graph is $(A,2)$-flow-choosable by Lemma~\ref{subgroup-flow-choosable-lem} with $k=2$.  In either case $|A|>3$, so $A$ satisfies Definition~\ref{admissible-group-def}, by Lemma~\ref{local-subdivision-choosability-lem} with $r=1$, $\ell=3$, and $k=1$.
\end{proof}

\begin{cor}
    \label{large-group-admissible-cor}
    Every finite abelian group of order at least $63$ is admissible.
\end{cor}
\begin{proof}
    If $|A|\ge 65$, then this holds by  Corollary~\ref{admissible-cor}.  If $|A|=63$, then $A$ has a subgroup of order $7$, of index $9$.  If $|A|=64$, then $A$ has a subgroup of order $8$, of index $8$.  In both exceptional cases, we are done by Corollary~\ref{admissible-cor}.
\end{proof}

\subsection{Flow reconfiguration via complementary subgroups}

The next lemma uses a decomposition $A=A_1\oplus A_2$ into complementary subgroups. Recall that $A=A_1\oplus A_2$ means that every element of $A$ can be written uniquely as $a_1+a_2$ with $a_i\in A_i$; equivalently, $A=A_1+A_2$ and $A_1\cap A_2=\{0_A\}$. Adding an $A_i$-flow preserves every coset of $A_i$ for $i\in\{1,2\}$, and the proof repeatedly exploits this property. This is equivalent to $A\cong A_1\times A_2$, but we will work throughout with the internal direct sum notation $A_1\oplus A_2$ for clarity.

\begin{lem}[Complementary subgroup method]
    \label{2group-flow-thm}
    Let $G$ be a $4$-edge-connected graph, and let $A$ be a finite abelian group,  where $A=A_1\oplus A_2$.
    If $A_1$ and $A_2$ are admissible subgroups with $|A_1|,|A_2|\ge 3$, then $G$ is $A$-flow-connected.
\end{lem}
\begin{proof}
Let $n:=|V(G)|$. Since $A_1$ and $A_2$ are admissible, $G$ admits nowhere-zero $A_1$- and $A_2$-flows $f_1'$ and $f_2'$. Let $f':=f_1'+f_2'$; note that $f'$ is nowhere-zero since $A_1\cap A_2=\{0_A\}$. We show that every nowhere-zero $A$-flow is equivalent to $f'$ by induction on $n$.

Let $f$ be a nowhere-zero $A$-flow, and let $E_1:=\{e\in E(G): f(e)\in A_1\}$ and $E_2:=\{e\in E(G): f(e)\in A_2\}$, with $E_0:=E_1\cup E_2$. Since $A_1\cap A_2=\{0_A\}$, the sets $E_1$ and $E_2$ are disjoint. If $|E_0|\ge n$, then the subgraph induced by $E_0$ contains a cycle $C$. 
We orient $C$ cyclically 
Choose an arbitrary edge $e\in E(C)$, 
and pick elements $a_1\in A_1\setminus\{0_{A_1},-f(e)\}$ and $a_2\in A_2\setminus\{0_{A_2},-f(e)\}$. 
This is possible because $|A_1|,|A_2|\ge 3$. Adding $a_1+a_2$ along $C$ preserves flow conservation and keeps all edges nonzero, while strictly decreasing $|E_0|$, since we remove $e$. Hence we assume that $|E_0|<n$.

Since $E_1$ and $E_2$ are disjoint, one of them has size less than $n/2$; by symmetry we assume that $|E_2|<n/2$. Let $E':=E(G)\setminus E_2$. Since $G$ is $4$-edge-connected, $|E(G)|\ge 2n$, and thus $|E'|>3n/2$. By Corollary~\ref{sparse-good-subgraph-unified}(1), the subgraph induced by $E'$ contains a subgraph $H$ that is either a cycle of length at most $3$ or a $(\le 1)$-subdivision of a $3$-edge-connected graph. In particular, $f(e)\notin A_2$ for every $e\in E(H)$.

We contract $H$ to a single vertex and obtain a graph $G^*$. Let $f^*$ and $(f')^*$ be the induced flows on $G^*$. By the induction hypothesis, $f^*$ and $(f')^*$ are equivalent in $\mathcal{F}(G^*,A)$, so there exists a reconfiguration sequence in $G^*$ transforming $f^*$ into $(f')^*$.

We lift this sequence to $G$. Each cycle in $G^*$ corresponds to a cycle in $G$ obtained by replacing occurrences of the contracted vertex with paths in the connected subgraph $H$. Throughout this process, we maintain the invariant that $f(e)\notin A_2$ for every $e\in E(H)$. To achieve this, before each lifted move adding some value $\lambda\in A$ along a lifted cycle $C$, we first modify the flow on $H$ by adding an $A_2$-flow $g_2$ so that every edge of $H$ has flow value outside $A_1$. For each edge $e\in E(H)$, the set $\{x\in A_2: f(e)+x\in A_1\}$ has size at most 1, since otherwise 2 such elements would differ by an element of $A_1\cap A_2=\{0_A\}$. As $A_2$ is admissible and $H$ is $(A_2,1)$-flow-choosable, such an $A_2$-flow $g_2$ exists. Decomposing $g_2$ into cycle-flows and adding them one-by-one preserves cosets of $A_2$ and ensures that no edge becomes 0. So after this step we have $f(e)\notin A_1\cup A_2$ for all $e\in E(H)$.

Next, we add an $A_1$-flow on $H$ to prepare for adding $\lambda$ along $C$. Fix a cyclic orientation of $C$, and for each edge $e\in E(H)$ let $\varepsilon_e\in\{-1,0,1\}$ denote its signed incidence in $C$, where $\varepsilon_e=1$ if $e$ is traversed by $C$ in the forward direction, $\varepsilon_e=-1$ if it is traversed in the reverse direction, and $\varepsilon_e=0$ if $e\notin E(C)$. For each $e\in E(H)$, the set $\{x\in A_1: f(e)+x+\varepsilon_e\lambda\in A_2\}$ has size at most 1. Since $A_1$ is admissible and $H$ is $(A_1,1)$-flow-choosable, we obtain such an $A_1$-flow $g_1$ so that every edge of $H$ has flow value outside $A_2$. Decomposing $g_1$ into cycle-flows and adding them one-by-one preserves cosets of $A_1$ and maintains the nowhere-zero condition. After this preparation, adding $\lambda$ along $C$ is valid both inside and outside $H$, and the invariant $f(e)\notin A_2$ on $H$ is restored.

Repeating this procedure for each step of the sequence in $G^*$ yields a nowhere-zero $A$-flow $r$ on $G$ that agrees with $f'$ outside $H$ and satisfies $r(e)\notin A_2$ for every $e\in E(H)$. Since $f'(e)\notin A_2$ for all $e\in E(G)\setminus E(H)$, the flow $r$ is nowhere-zero on all edges.

Finally, since $A=A_1\oplus A_2$, we decompose $f'-r$ as $h_1+h_2$, where $h_1$ is an $A_1$-flow and $h_2$ is an $A_2$-flow. We first add the cycle decomposition of $h_2$ to $r$, which preserves cosets of $A_2$ at every step and therefore maintains the nowhere-zero property. The resulting flow $r'$ has the same $A_2$-component as $f'$ and thus $r'(e)\notin A_1$ for all $e\in E(H)$. We then add the cycle decomposition of $h_1$; this preserves cosets of $A_1$ at every step and again maintains the nowhere-zero property. The resulting flow is exactly $f'$, as desired.
\end{proof}

\subsection{Large Cyclic Factors}

We next handle groups with a large cyclic factor, $\Z_p$.  Our proof first prepares the $\Z_p$-coordinates so that they avoid a fixed interval around $0$; after that, the $\Z_p$-coordinates can be reconfigured directly, and the other coordinates can be changed freely.

\begin{cor}
    \label{local-Z28-cor}
    Let $H$ be an orientation of either a $(\le 5)$-cycle or a $(\le 2)$-subdivision of a $3$-edge-connected graph.  For every function $\psi:E(H)\to\Z_{28}$, there exists a $\Z_{28}$-flow $f$ on $H$, zero values allowed, such that
    \[
        f(e)\ne \psi(e)
        \qquad\text{for every }e\in E(H).
    \]
\end{cor}
\begin{proof}
    Since $\Z_{28}$ has a subgroup of order $4$ and index $7$, Lemma~\ref{subgroup-flow-choosable-lem} with $k=3$ implies that every $3$-edge-connected graph is $(\Z_{28},3)$-flow-choosable.  The result now follows from Lemma~\ref{local-subdivision-choosability-lem} with $A=\Z_{28}$, $r=2$, $\ell=5$, and $k=1$.
\end{proof}

\begin{cor}
    \label{local-28-flow-patch-cor}
    Let $H$ be an orientation of either a $(\le 5)$-cycle or a $(\le 2)$-subdivision of a $3$-edge-connected graph.  For every function
    \[
        \psi:E(H)\to[-27,27],
    \]
    there exists an integer $28$-flow $h$ on $H$, zero values allowed, such that
    \[
        h(e)\ne \psi(e)
        \qquad\text{for every }e\in E(H).
    \]
\end{cor}
\begin{proof}
    Apply Corollary~\ref{local-Z28-cor} to the residues of $\psi$ modulo $28$, and then lift the resulting $\Z_{28}$-flow to an integer $28$-flow using Lemma~\ref{Zk-k-equiv-lem}.  Congruence avoidance modulo $28$ implies integer avoidance because both $h(e)$ and $\psi(e)$ lie in $[-27,27]$.
\end{proof}

\begin{lem}
    \label{away-from-boundary-lem}
    Let $A$ be an abelian group.  Let $p$ and $q$ be coprime positive integers with
        $q\ge 55$
        {and}
        $p>q^2\cdot 27+27$.
    Let
        $I:=\{-27,-26,\ldots,26,27\}\subseteq\Z_p.$
    If $G$ is $4$-edge-connected, then for every nowhere-zero $(\Z_p\times A)$-flow $f$ in $G$, there exists a nowhere-zero $(\Z_p\times A)$-flow $g$ in $G$ such that $f$ and $g$ lie in the same component of $\F(G,\Z_p\times A)$ and
       $ g(E(G))\cap (I\times A)=\emptyset.$
\end{lem}
\begin{proof}
For a flow $x$, write $x(e)=(x_p(e),x_A(e))$. For an integer $s$ and a subgraph $H\subseteq G$, we say that $x$ is \emph{$s$-clean on $H$} if $q^s x_p(e)\notin I$ for every $e\in E(H)$. We say that $x$ is \emph{$s$-safe on $H$} if, for every $e\in E(H)$, either $x_A(e)\neq 0_A$ or $q^s x_p(e)\notin I$. Thus $s$-clean implies $s$-safe.
We first prove a local adjustment claim.

\medskip
\noindent\textbf{Claim.}
Let $H$ be either a cycle of length at most $5$ or a $(\le 2)$-subdivision of a $3$-edge-connected graph. Let $x$ be a nowhere-zero $(\mathbb{Z}_p\times A)$-flow that is $s$-safe on $H$. Let $C$ be either a cycle of $G$ or the empty subgraph, and let $\lambda=(\lambda_p,\lambda_A)\in\mathbb{Z}_p\times A$. Assume that for every edge $e\in E(G)\setminus E(H)$, the flow obtained after adding $\lambda$ on $C$ remains nonzero on $e$ in $\mathbb{Z}_p\times A$. Then, by first using only cycle reconfigurations supported inside $H$, and then adding $\lambda$ on $C$, we can reach a nowhere-zero flow $y$ that is $(s-1)$-clean on $H$.

\medskip
\noindent\emph{Proof of the claim.}
Fix a cyclic orientation of $C$ if $C\neq\emptyset$. For each $e\in E(H)$, let $\varepsilon_e\in\{-1,0,1\}$ be $0$ if $e\notin E(C)$, and otherwise the sign with which $e$ is used by the cyclic orientation of $C$. Let $x_e:=q^s(x_p(e)+\varepsilon_e\lambda_p)\in\mathbb{Z}_p$.

We call $z\in I$ \emph{bad} for $e$ if $x_e+z\in qI$. Each edge has at most 1 bad value. Indeed, if $z_1,z_2\in I$ are both bad, then there exist $r_1,r_2\in I$ such that $x_e+z_i\equiv q r_i\pmod p$ for each $i\in[2]$. Subtracting these congruences gives $z_1-z_2\equiv q(r_1-r_2)\pmod p$.
Since $|z_1-z_2-q(r_1-r_2)|\le 54(q+1)<p$, the above congruence must in fact be an equality over the integers. Because $q\ge 55$ and $|z_1-z_2|\le 54$, this forces $r_1=r_2$, and hence $z_1=z_2$.

Thus each edge admits at most one forbidden value in $I$. Let $\psi(e)$ be this value if it exists, and choose $\psi(e)\in I$ arbitrarily otherwise. By Corollary~\ref{local-28-flow-patch-cor}, there exists a $28$-flow $h$ on $H$ such that $h(e)\neq \psi(e)$ for every $e\in E(H)$.
We next decompose $h$ as in Lemma~\ref{decompose-flow-lem-II}, so that every partial sum $h'$ satisfies $h'(e)\in I$ for every edge $e\in E(H)$.

Starting from $x$, we add the $\mathbb{Z}_p$-flow $q^{-s}h$ to the first coordinate on $H$, one cycle at a time according to this decomposition. During these internal moves, the $A$-coordinate on $H$ remains unchanged. Moreover, nowhere-zero-ness is preserved: if $x_A(e)\neq 0_A$, then $e$ stays nonzero; if $x_A(e)=0_A$, then $s$-safety guarantees $q^s x_p(e)\notin I$, 
and adding any partial contribution in $I$ cannot create a zero first coordinate.  This final claim merits a bit more explanation.  Suppose that $x_p(e)+\sum_{i=1}^{\ell}q^{-s}h_i(e)\equiv0\pmod p$ for some choice of $\ell$.
Now $q^sx_p(e)\equiv-\sum_{i=1}^{\ell}h_i(e)\in I$, a contradiction.  (The point is that $I$ contains all partial sums of these $h_i$'s, and the negatives of these sums, which contradicts our assumption that $x_p(e)$ is $s$-safe.)

After completing these internal adjustments, we add $\lambda$ on $C$. This step is safe outside $H$ by assumption. For $e\in E(H)$, the resulting flow $y$ satisfies $q^s y_p(e)=x_e+h(e)$. Since $h(e)\neq \psi(e)$ for every $e$, this implies $q^s y_p(e)\notin qI$, and hence $q^{s-1}y_p(e)\notin I$. Therefore $y$ is $(s-1)$-clean on $H$, as required. This completes the proof of the claim.
\hfill $\lozenge$

\medskip

We now proceed with the main proof by induction on $|V(G)|$. The case $|V(G)|=1$ is immediate.

For $s\in\{0,1,2\}$, we say an edge $e$ is $s$-safe if either $f_A(e)\neq 0_A$ or $q^s f_p(e)\notin I$. If $f_A(e)\neq 0_A$, then $e$ is $s$-safe for all three choices of $s$. If $f_A(e)=0_A$, then $f_p(e)\neq 0$ since $f$ is nowhere-zero. We claim that at most one $s\in\{0,1,2\}$ satisfies $q^s f_p(e)\in I$. Indeed, if $q^{s_1}f_p(e)=r_1\in I$ and $q^{s_2}f_p(e)=r_2\in I$ with $s_1<s_2$, then letting $d=s_2-s_1\in\{1,2\}$ we obtain $q^d r_1-r_2\equiv 0\pmod p$. Since $|q^d r_1-r_2|\le q^2\cdot 27+27<p$, this congruence must be an equality over the integers, which is impossible unless $r_1=0$, contradicting $f_p(e)\neq 0$. Hence each edge is safe for at least two values of $s$.

Therefore there exists $s_0\in\{0,1,2\}$ such that the set $E'$ of $s_0$-safe edges satisfies $|E'|\ge \frac{2}{3}|E(G)|$. Since $G$ is $4$-edge-connected, we have $|E(G)|\ge 2|V(G)|$, and hence $|E'|\ge \frac{4}{3}|V(G)|$.

Let $G'$ be the subgraph induced by $E'$, after deleting isolated vertices. If $|V(G')|\le 2$, then $G'$ contains a $2$-cycle. Otherwise, since $|E(G')|=|E'|\ge \frac{4}{3}|V(G)|>\frac{4}{3}|V(G')|-2$, we apply Corollary~\ref{sparse-good-subgraph-unified}(2). Hence in all cases, $G'$ contains a subgraph $H$ that is either a $(\le 5)$-cycle or a $(\le 2)$-subdivision of a $3$-edge-connected graph.

Apply the claim with $x=f$, $s=s_0$, $C=\emptyset$, and $\lambda=0$. We obtain a flow $f_0$ equivalent to $f$ that is $(s_0-1)$-clean on $H$.

Contract $H$ to a single vertex, delete loops, and call the resulting graph $G^*$. Then $G^*$ is again $4$-edge-connected. Let $f^*$ be the restriction of $f_0$ to $G^*$. By induction, $f^*$ is equivalent in $\mathcal{F}(G^*,\mathbb{Z}_p\times A)$ to a flow $g^*$ satisfying $g^*(E(G^*))\cap (I\times A)=\emptyset$.

We lift a reconfiguration sequence from $f^*$ to $g^*$ step by step. Suppose a step in $G^*$ adds $\lambda_i\in\mathbb{Z}_p\times A$ on a cycle $C_i^*$. We lift $C_i^*$ to a cycle $C_i$ in $G$ by replacing passages through the contracted vertex with paths inside $H$. If the current flow is $s$-clean on $H$, then it is $s$-safe on $H$, so the claim applies with $C=C_i$ and $\lambda=\lambda_i$. Outside $H$, safety follows from the corresponding step in $G^*$. Hence each lifted step preserves the controlled decrease of cleanliness on $H$.

After lifting the entire sequence, we obtain a flow $y$ that agrees with $g^*$ outside $H$ and is $s$-clean on $H$ for some positive integer $s$ (positivity uses that $q^{\phi(p)}\equiv1\pmod p)$.

Finally, we apply the claim repeatedly with $C=\emptyset$ and $\lambda=0$. Each application reduces the cleanliness level on $H$ by 1 without affecting edges outside $H$. Repeating this process $s$ times yields a flow $g$ such that $g_p(e)\notin I$ for all $e\in E(H)$. Since $g$ already avoids $I\times A$ outside $H$, we conclude that $g(E(G))\cap (I\times A)=\emptyset$.

Every step preserves the component of $\mathcal{F}(G,\mathbb{Z}_p\times A)$, so $g$ is equivalent~to~$f$. 
\end{proof}

\begin{lem}
\label{away-connected-lem}
Fix an integer $p\ge 17$, let $G$ be an oriented $4$-edge-connected graph, and let
$J:=\{-6,-5,\ldots,5,6\}\subseteq \mathbb{Z}_p$.
If $f$ and $g$ are $\mathbb{Z}_p$-flows in $G$ such that $f(E(G))\cup g(E(G))$ is disjoint from $J$, then $f$ and $g$ lie in the same component of $\mathcal{F}(G,\mathbb{Z}_p)$.
\end{lem}

\begin{proof}
Since $G$ is $4$-edge-connected, it admits a nowhere-zero integer $4$-flow $h'$. 
Let $h(e):=h'(e)\pmod 4$ for all $e\in E(G)$.
Let $f,g$ be $\mathbb{Z}_p$-flows in $G$ such that $f(E(G))\cup g(E(G))$ is disjoint from $J$. For each of $f$ and $g$, choose integer flows $f',g'$ satisfying $f'(e)\equiv f(e)\pmod p$ and $g'(e)\equiv g(e)\pmod p$ with $0<|f'(e)|,|g'(e)|<p$ for all edges $e$. Because $f$ and $g$ avoid $J$, we have $6<|f'(e)|,|g'(e)|<p-6$ for every edge $e$.

\begin{figure}[!h]
\centering
\begin{tikzpicture}[node distance=0.1cm and 0.325cm, scale=1.2]
  
  \node (fprime) {$f'$};
  \node (lr1)    [right=of fprime] {$\longleftrightarrow$};
  \node (term1)  [right=of lr1]    {$f'-f_4'+ h'$};
  \node (lr2)    [right=of term1]  {$\longleftrightarrow$};
  \node (hprime) [right=of lr2]    {$h'$};
  \node (lr3)    [right=of hprime] {$\longleftrightarrow$};
  \node (term2)  [right=of lr3]    {$g' -g_4'+h'$};
  \node (lr4)    [right=of term2]  {$\longleftrightarrow$};
  \node (gprime) [right=of lr4]    {$g'$};

  \node (arrow1) [below=0.1cm of fprime] {$\uparrow$};
  \node (arrow2) [below=0.1cm of gprime] {$\uparrow$};

  \node (f) [below=0.1cm of arrow1] {$f$};
  \node (g) [below=0.1cm of arrow2] {$g$};

\end{tikzpicture}
	\captionsetup{width=.720\textwidth}
	\caption{To show that $\mathbb{Z}_p$-flows $f$ and $g$ are equivalent, we consider corresponding $p$-flows $f'$ and $g'$.  We show that $f'$ and $g'$ are each equivalent (as $p$-flows) to a common $p$-flow $h'$ (which is also a $4$-flow).%
}
\end{figure}

We first transform $f'$ to $f'-f'_4+h'$, where $f'_4$ is an integer-valued function satisfying $f'_4(e)\equiv f'(e)\pmod 4$ and $|f'_4(e)|\le 3$ for all $e$. Note that $h'-f'_4$ is an integer flow with $|(h'-f'_4)(e)|\le 6$ for every edge $e$.
By Lemma~\ref{decompose-flow-lem-II}, $h'-f'_4$ can be written as a sum of integer flows $(\psi_i)_{i=1}^s$, each supported on a cycle, such that for every edge $e$ and every $j\in[s]$, the partial sum $\sum_{i\le j}\psi_i(e)$ lies between $0$ and $(h'-f'_4)(e)$.

We start from $f'$ and successively add the flows $\psi_1,\ldots,\psi_s$. At each step, only edges on a single cycle are modified, and in total all steps change every affected edge by at most $6$ in absolute value. Since initially $6<|f'(e)|<p-6$, all intermediate values remain nonzero and stay in $(-p,p)$; hence, they define valid $\mathbb{Z}_p$-flows. This yields $f'-f'_4+h'$.

Next we transform $f'-f'_4+h'$ to $h'$.  The difference is $f'-f'_4$, which is divisible by $4$.  Decompose $(f'-f'_4)/4$ as in Lemma~\ref{decompose-flow-lem-II}, and subtract four times the partial cycle-flows.  Every intermediate value is congruent to $h'(e)$ modulo $4$, and hence is nonzero.  It also lies between $h'(e)$ and $(f'-f'_4+h')(e)$, so its absolute value is less than $p$.

Thus, $f'$ is equivalent to $f'-f_4'+h'$, which is equivalent to $h'$.  Similarly, $g'$ is also equivalent to $h'$.  Hence, by transitivity, $f'$ is equivalent to $g'$, as nowhere-zero $p$-flows.  Thus, by Lemma~\ref{k-Zk-hom-lem} we get that $f$ and $g$ are equivalent as $\Z_p$-flows.
\end{proof}

\begin{cor}
    \label{coprime-cor}
    Let $A$ be an abelian group, and let $p,q$ be coprime positive integers with
    \[
        q\ge55
        \qquad\text{and}\qquad
        p>q^2\cdot27+27.
    \]
    If $G$ is a $4$-edge-connected graph, then $\F(G,\Z_p\times A)$ is connected.
\end{cor}
\begin{proof}
    Let $I:=[-27,27]\subseteq\Z_p$, and let $f,g$ be nowhere-zero $(\Z_p\times A)$-flows on $G$.  By Lemma~\ref{away-from-boundary-lem}, we may reconfigure $f$ and $g$ to flows
    $%
        f'=(f_1',f_2')%
    $
    and 
    $%
        g'=(g_1',g_2')%
    $
    whose values avoid $I\times A$.
    Since $I$ contains $[-6,6]$, the $\Z_p$-flows $f_1'$ and $g_1'$ are nowhere-zero and avoid $[-6,6]$.  By Lemma~\ref{away-connected-lem}, $f_1'$ and $g_1'$ are equivalent in $\F(G,\Z_p)$.

    The difference $g_2'-f_2'$ is an $A$-flow.  We decompose this flow into cycle-flows and add these cycle-flows,
    one-by-one, to the second coordinate of $(f_1',f_2')$.  All intermediate flows are nowhere-zero because the first coordinate $f_1'$ is nowhere-zero.  This connects $(f_1',f_2')$ to $(f_1',g_2')$.  Now we lift a reconfiguration sequence from $f_1'$ to $g_1'$ in the first coordinate while keeping the second coordinate equal to $g_2'$.  This connects $(f_1',g_2')$ to $(g_1',g_2')=g'$.  Therefore, $f$ and $g$ are equivalent.
\end{proof}

\begin{cor}
    \label{ZpxA-cor}
    Let
    \[
        P_0:=56^2\cdot27+27=84,699.
    \]
    For every prime power $p>P_0$, every finite abelian group $A$, and every $4$-edge-connected graph $G$, the graph $\F(G,\Z_p\times A)$ is connected.
\end{cor}
\begin{proof}
    Since $55$ and $56$ are coprime, at least one of them is coprime with $p$.  Choose $q\in\{55,56\}$ with $\gcd(p,q)=1$.  Now
    \[
        p>P_0\ge q^2\cdot27+27,
    \]
    so we are done by Corollary~\ref{coprime-cor}.
\end{proof}

\subsection{Proof of the Main Theorem}

\begin{proof}[Proof of Theorem~\ref{4-edge-connected-thm}]
Let $P_0=56^2\cdot 27+27$. Write the finite abelian group $A$ as
$A\cong \Z_{a_1}\times\cdots\times\Z_{a_s}$, where each $a_i$ is a prime power.

If some $a_i>P_0$, then we write $A\cong \Z_{a_i}\times A'$ for some abelian group $A'$. In this case the conclusion follows from Corollary~\ref{ZpxA-cor}. Hence we assume that $a_i\le P_0$ for all $i$.

We now construct a direct factor $B$ of $A$ such that $63\le |B|\le P_0$. If some factor $\Z_{a_i}$ satisfies $a_i\ge 63$, then we let $B:=\Z_{a_i}$. Otherwise every $a_i\le 62$, and we take $B$ to be a minimal subproduct of the factors whose order is at least $63$. Since each factor has order at most $62$, the order of $B$ is at most $62^2$; in particular, $|B|<P_0$.

Let $C$ be the complementary direct factor of $B$, so that $A\cong B\times C$. Since $|A|\ge 62P_0+1$ and $|B|\le P_0$, we have $|C|=|A|/|B|>62$; hence, $|C|\ge 63$.

By Corollary~\ref{large-group-admissible-cor}, both $B$ and $C$ are admissible groups. Viewing $A$ as $B\times C$, Lemma~\ref{2group-flow-thm} implies that $G$ is $A$-flow-connected.
\end{proof}

{\small
\bibliographystyle{habbrv} 
\bibliography{references}
}
 
\end{document}